\documentclass[a4paper,11pt]{article}
\usepackage[latin1]{inputenc}
\usepackage[T1]{fontenc}
\usepackage[british,UKenglish,USenglish,american]{babel}
\usepackage{amsmath}
\usepackage{amsfonts}
\usepackage{hyperref}
\usepackage[T1]{fontenc}
\usepackage[latin1]{inputenc}
\usepackage{ntheorem}
\usepackage{theorem}
\usepackage{graphics}
\usepackage{makeidx}
\usepackage{fancyhdr}
\usepackage{hyperref}
\usepackage{verbatim}
\usepackage{bbm}
\usepackage{bm}
\usepackage{amssymb}
\usepackage{color}
\numberwithin{equation}{section}
\usepackage{amssymb}
\newcommand{\be}{\begin{equation}}
\newcommand{\ee}{\end{equation}}
\newcommand{\benn}{\begin{equation*}}
\newcommand{\eenn}{\end{equation*}}
\newcommand{\bea}{\begin{eqnarray}}
\newcommand{\eea}{\end{eqnarray}}

\newcommand{\beann}{\begin{eqnarray*}}
\newcommand{\eeann}{\end{eqnarray*}}

\newtheorem{theorem}{Theorem}[section]

\newtheorem{lemma}[theorem]{Lemma}

\newtheorem{definition}[theorem]{Definition}
\theorembodyfont{\rm}

\newtheorem{remark}[theorem]{Remark}

\newtheorem{assumptions}[theorem]{Assumptions}
\newtheorem{ansatz}[theorem]{Ansatz}

\newcommand{\qed}{\hfill $\Box$\smallskip}
\newcommand{\E}{\noindent{$\mathbb{E}$ \ }}

\def\R{\mathbb{R}}

\def\N{\mathbb{N}}

\def\P{\mathbb{P}}
\def\E{\mathbb{E}}

\def\P{\mathbb{P}}

\def\cF{\mathcal{F}}

\def\cH{\mathcal{H}}

\def\cL{\mathcal{L}}

\def\cN{\mathcal{N}}
\def\cO{\mathcal{O}}

\def\cS{\mathcal{S}}

\def\txtd{{\textnormal{d}}}

\def\txtD{{\textnormal{D}}}

\allowdisplaybreaks

\title{Bifurcation theory for SPDEs: finite-time Lyapunov exponents and amplitude equations}
\author{Dirk Bl\"omker\thanks{Institut f\"ur Mathematik, Universit\"at Augsburg, Universit\"atsstra{\ss}e 12, 86135 Augsburg, Germany.~E-Mail: dirk.bloemker@math.uni-augsburg.de}~~~and~~Alexandra Neam\c tu\thanks{ Department of Mathematics and Statistics, University of Konstanz,
Universit\"atsstra\ss{}e~10, 78464 Konstanz, Germany. E-Mail: alexandra.neamtu@uni-konstanz.de} }
\begin{document}
\maketitle

\begin{abstract}
We consider a stochastic partial differential equation close to bifurcation of pitchfork type, where a one-dimensional space changes its stability.  
For finite-time Lyapunov exponents we characterize regions depending on the distance from bifurcation and the noise strength where finite-time Lyapunov exponents are positive and thus detect changes in stability. 

One technical tool is the reduction of the 
essential dynamics of the infinite dimensional stochastic system to a simple ordinary stochastic differential equation, 
which is valid close to the bifurcation. 
\end{abstract}

\section{Introduction}

The main goal of this work is to 
provide an analysis of the stability
for general SPDEs of the type
\begin{align}\label{spde1}
\begin{cases}
\txtd u = [A u + \nu u +\cF(u)]~\txtd t +\sigma \txtd W_t \\
u(0)=u_{0},
\end{cases}
\end{align}
where $A$ is a linear operator with a one dimensional kernel, the parameter $\nu\in\R$ shifts the spectrum of $A$, $\cF$ is a stable cubic nonlinearity and $\sigma>0$ denotes the intensity of the infinite-dimensional noise, which is given by a Hilbert-space valued Wiener process $(W(t))_{t\in[0,T]}$. 

For the deterministic PDE, i.e.\ for $\sigma=0$, we expect a classical pitchfork-type bifurcation at $\nu=0$, which is easy to verify  in many examples. 
For the stochastic case our analysis relies on finite-time Lyapunov exponents (FTLE), which turned out to be useful to detect bifurcation points for stochastic (partial) differential equations~\cite{R,BlEnNe:21}. 
Finite-time Lyaponov exponents measure the pathwise local expansion rate of nearby solutions.  
This is not limited to statistically invariant solutions, but can be evaluated at any solution of the SPDE.
If the FTLE are positive for a given time and a given solution, then nearby solutions tend to separate on a finite time horizon given by that time. 
In contrast to that 
negative finite-time Lyapunov exponents indicate a pathwise local attraction on a given time-scale.

In our main results we study FTLE for various initial conditions and various values of $\nu$ and $\sigma$.
A key technical idea is to use the reduction via amplitude equations close to a change of stability. 
This allows to reduce the infinite dimensional dynamics to a finite dimensional SDE, which is a center like part, although this finite dimensional space is not invariant for the dynamics of the SPDE. 
A particular example we will study is the solution of the SPDE starting in a suitably chosen initial condition, so that the amplitude equation will evolve according to a statistically invariant solution given by a random attractor.


\paragraph{Main Results.}
We show the following results for FTLE  of~\eqref{spde1} near $\nu=0$, which is the bifurcation point of the deterministic PDE.  Below we refer to $\nu=0$ as the bifurcation point.

\begin{itemize}
    \item[I.] {\em Before the bifurcation, $\nu<0$.}
    The solutions of~\eqref{spde1} are stable for all $\sigma$ with probability one. 
    
    To be more precise, following well-known results we show in Theorem~\ref{l:neg} that FTLEs for all solutions and all times are negative with probability one.

    \item[II.] {\em After the bifurcation, moderate noise strength,   $0<\nu\approx\sigma \ll1$.}
    Here we have instability, in the sense that there is at least one time $T$ and one solution for which the finite-time Lyapunov exponent $ \lambda_T>0 $ is positive with positive probability.
    
    The proof of this statement (Theorem~\ref{thm:main1}) relies on the approximation with the amplitude equation given by the SDE evolving on a slow time scale $T$ given by $\txtd  b =( b -b^3)~\txtd T +\frac{\sigma}{\nu} ~\txtd \beta(T)$, where $(\beta(T))_{T\geq 0}$ is a Brownian motion, and we  consider the SPDE starting in the rescaled random attractor of the amplitude equation.   

   Due to the nature of the approximation result (Theorem~\ref{thm:aprox}), our result is applicable for times $T$ of order $1/\nu$ but for technical reasons not up to $0$.  This means there is a $\nu$ dependent interval, where we can prove that the FTLE is positive and this interval contains values of the type $C/\nu$ for some values of the constant $C>0$. 
    
\item[III.] {\em After the bifurcation, small noise strength, $1\gg \nu \gg \sigma \geq0$.} 
Similar to case II.\ we observe instability for at least one solution and one time. 

This relies on the approximation via the amplitude equation $\txtd b = (b - b^3)~\txtd T$. We show (Theorem~\ref{thm:main2}) that for the solution with initial condition zero, we have a positive FTLE $ \lambda_T>0 $ for times $T$ of order $1/\nu$  with positive probability. But this result should hold true for any solution with a sufficiently small initial condition.

    \item[IV.] {\em At the bifurcation, $0\leq \nu \ll \sigma \ll1$.}
    Here we have
    stability, meaning that $\lambda_T<0$  for all solutions with positive probability. The proof of this statement (Theorem~\ref{thm:main3}) relies on the approximation with 
    the amplitude equation  of the type  $\txtd b = -b^3\txtd T + \txtd \beta(T)$, where $(\beta(T))_{T\geq 0}$ is a Brownian motion.  
\end{itemize}

Let us remark that in II.\ the probability where the finite time Lyapunov exponent is positive might be very small, as the proof relies on trajectories where the random attractor is initially close to $0$. In contrast to that in III.\ and IV.\ the probability is significantly larger and close to one, but we do not quantify this in more detail. 

\begin{remark}
We emphasize that based on the interplay between $\nu$ (the bifurcation parameter) and $\sigma$ (the intensity of the noise) as given above in the cases II.-IV., we can estimate the FTLEs of the SPDE~\eqref{spde1} under the Assumptions~\ref{a},~\ref{f} and~\ref{s} on the coefficients and on a specified time scale.
To this aim we use the order of the approximation of the SPDE with the amplitude equation (Theorem~\ref{thm:aprox}) to control the error term between the linearization of the SPDE and the linearization of the amplitude equation. In all cases we state bounds on the corresponding FTLEs for the SDEs but do not directly estimate the difference 
between the FTLEs for the SPDEs and the SDEs, as the FTLE of the SDE is in general not sufficient to control the linearization of the SPDE.  
\end{remark}

\subsection{Known results}

Lyapunov exponents and amplitude equations are well known tools to study random dynamical systems. We briefly summarize results related to our work.

\paragraph{Lyapunov exponents.}
The famous work of Crauel-Flandoli~\cite{CrFl:98} establishes that the attractor of the SDE
\begin{align}\label{sde} \txtd x = (\alpha x -x^3)~\txtd t+\sigma~\txtd W_t 
\end{align}
is a unique random fixed point independent of the bifurcation parameter $\alpha\in\R$ and of the noise intensity $\sigma>0$, inferring that additive noise destroys the deterministic pitchfork bifurcation. This phenomenon is also referred to as {\em synchronization by noise} and can be read from the top asymptotic Lyapunov-exponent, which is always negative.~An analogous statement holds for the Allen-Cahn SPDE~\cite{C}, whereas further statements on synchronization by noise for order-preserving dynamical systems can be looked up in~\cite{FlGeSch:17}. A recent result which goes beyond order-preservation is available in \cite{GeTs:22}, where the invariant measure is used to determine the Lyapunov exponent in a system of reaction diffusion equations.

However,
Callaway et al.~\cite{C} (based on an earlier work of Rasmussen~\cite{R:10}) challenged this point of view by measuring local stability for trajectories of
the SDE~\eqref{sde} on finite time scales, using finite-time Lyapunov exponents. They proved that, whereas all FTLEs are negative for $ \alpha < 0,$ there is always a positive
probability to observe positive FTLEs for $\alpha>0$, i.e.\ sensitivity of initial conditions for finite time
dynamics; this change of the FTLEs corresponds with a transition from uniform to non-uniform
attractivity of the attractor and a loss of (uniform) hyperbolicity, a main paradigm of bifurcation
theory. In order to prove such statements, a detailed analysis on the location of the random attractor
is required. These results have been extended to the stochastic Allen-Cahn equation in~\cite{BlEnNe:21} using a novel argument for cones invariance.

Our result in contrast to the previously stated result does not require the monotonicity of the SPDE. We carry over  properties of the SDE via amplitude equations to the SPDE.

Finally, we remark that it is a very challenging task to obtain sign information on {\em asymptotic Lyapunov exponents}. In a series of remarkable papers J.~Bedrossian, A.~Blumenthal and S.~Punshon-Smith~\cite{BeBlPs2:22,BeBlPS:21} introduced methods based on the Fisher information for obtaining strictly positive lower bounds on the top Lyapunov exponent of high-dimensional, stochastic differential equations such as the weakly damped Lorenz-96 (L96) model or Galerkin truncations of the 2D Navier-Stokes equations. The positivity of the top Lyapunov exponent has been also obtained for Lagrangian flows in~\cite{BeBlPS:22}.
 We emphasize that our analysis on the Lyapunov exponents is made on a suitable finite-time scale, which is motivated by our applications to bifurcation theory. \\

 An alternative but related problem is to fix the bifurcation parameter $\nu$
and introduce a coefficient $\varepsilon$ in front of the noise and determine for $\varepsilon\ll 1$ the extent to which the random motion resembles
that of the deterministic system, i.e.\ $\varepsilon=0$. For the stochastic Allen-Cahn equation, these situation has been for example investigated in~\cite{BeGe:13, Ba:15}. For the latter, the bulk of initial conditions relax to
one of two stable solutions, while when $\varepsilon>0$ these solutions are metastable, with typical random
trajectories transitioning from the vicinity of one to the other, reflecting the fact that the system
admits a unique stationary measure. Large deviations estimates provide a way of estimating the
typical timescale of these transitions, thereby giving another perspective on the "persistence" of
the bifurcation in the presence of noise.
 
  \paragraph{Amplitude equations.}
  These are a well-known tool to 
study stochastic dynamics for SPDEs close to a bifurcation. There are numerous results on the approximation of SPDEs via amplitude equations mostly treating results with Gaussian noise given by the derivative of a Wiener process, but we comment only on a few of them.
In~\cite{BHP} the authors study large domains, where the dominating pattern is slowly modulated in space and the amplitude equation 
is thus an SPDE of Ginzburg-Landau type. See also \cite{BiDBSch:19} for an example of a fully unbounded domain. Degenerate noise was studied in \cite{BloemkerMohammed} and several other publications. Here the noise is not acting directly on the dominating pattern, but through interaction of nonlinearity and
additive noise additional terms appear in the amplitude equation that  have the potential to stabilize the dynamics of the dominant modes. While all these references treated the case of additive Wiener noise, other types of noise were discussed in the literature.
Multiplicative Wiener noise was studied by \cite{BloemkerFu}, 
while for additive fractional noise with $H>1/2$ certain results 
for Rayleigh-B\'enard convection are available in~\cite{Bl:05}. Recently, fractional noise for $H\in(0,1)$ was considered in~\cite{BlNe:22} and the case of 
$\alpha$-stable L\'evy noise was treated in \cite{DBSY:22}.

While most approximation results treat large but finite time-scales only, the result on 
{\em multiscale expansion of invariant measures} from \cite{BlHa:04} gives an approximation of the invariant measure of a Swift-Hohenberg SPDE in terms of the invariant measure of the Amplitude Equation. This can be used for a result characterizing a P-bifurcation. For the definition and further details on this bifurcation see \cite{Ar:98}.
    
For SPDEs close to but still below the deterministic bifurcation one can observe {\em pattern formation below criticality}. Here the trivial solution is stable without adding noise, but using amplitude equations one can show for the perturbed model that patterns are visible for long times. 
See Section 3.2 of \cite{Bl:07}. This is an indication that noise shifts the deterministic bifurcation point. 
But let us point out that this is not captured by the finite time Lyapunov exponents, as we see in case I.

 \paragraph{Possible Extensions.}
 We plan to treat in a future work SPDEs with quadratic nonlinearities, such as Burger's equation. We expect to generalize the results obtained here for such SPDEs. However the order of the approximation result via amplitude equations will be slightly worse~\cite{Bl:05}. Moreover, we expect that multiplicative noise can be handled in the SPDE if the diffusion coefficient in front of the noise is non-zero, i.e.~$\sigma(u)\neq 0$ for all $u\neq 0$. Then similar results should hold around the bifurcation as presented here. If the   diffusion coefficient is zero, then multiplicative noise will appear in the amplitude equation, and the problem is wide open in that case even for an SDE. Adding higher order terms to the cubic nonlinearity does not change the amplitude equation, so similar results are expected in that case. Also, considering other stable nonlinearities like quintic instead of cubic should only change the technical details of the proof.
Challenging is the case when a two or higher dimensional space changes stability at the bifurcation, as we crucially rely on the properties of the one-dimensional amplitude equation, when we use the random attractor of the SDE or Birkhoff's ergodic theorem.
Even for two dimensional SDEs there is no result for FTLE available yet. 

\paragraph{Structure of the manuscript.} This work is structured as follows. In Section~\ref{assumptipns} we specify the setting of the problem, state the necessary assumptions and collect important properties of finite-time Lyapunov exponents. Section~\ref{sec:stability} is devoted to the most simple case that we consider, i.e.\ $\nu<0$. Here we trivially have stability, meaning that the finite-time Lyapunov exponents around the singleton attractor are always negative with probability one. 
Section~\ref{ly:a} discusses finite-time Lyapunov exponents for the two types of amplitude equations that we obtain. These results will be used later in the main Section~\ref{six}, where we approximate the Lyapunov-exponents of the SPDE~\eqref{spde1} with the ones of the amplitude equations. 
In Section~\ref{six} we derive error bounds for the linearization of the SPDE and of the amplitude equations in several situations regarding the bifurcation parameter $\nu$ and the intensity of the noise $\sigma$. The necessary approximation result of the SPDE via amplitude equations near a change of stability is derived in Section~\ref{five}. 
Finally, we provide in Section~\ref{seven} possible applications of our theory. These include the stochastic Allen-Cahn, Swift-Hohenberg and a surface growth model.

\section{Setting and Assumptions}\label{assumptipns}
We work in the following setting. We let $\cH$ stand for a separable Hilbert space and consider the SPDE
\begin{align}\label{spde}
\begin{cases}
\txtd u = [A u + \nu u +\cF(u)]~\txtd t +\sigma \txtd W_t \\
u(0)=u_{0}\in \cH.
\end{cases}
\end{align}
We make the following standard assumptions on the coefficients.
\begin{assumptions}\label{a} (Differential operator $A$) The linear operator $A$ generates a compact analytic semigroup $(e^{tA})_{t\geq 0}$ on $\cH$. 
 Moreover, it is symmetric and non-positive and has a {\em one-dimensional kernel} which we denote by $\cN$.  
 We define  the orthogonal projection  $P_c$ onto $\cN$, set $P_s=\text{Id}-P_c$ and obtain that $\cH=\cN\oplus \cS$, where $\cS$ stands for the range of $P_s$. 
 The semigroup is exponentially stable on $P_s\cH$ which means that there exists $\mu>0$ such that 
\[ 
\|e^{tA} P_s\|_{\cL(\cH)}\leq e^{-t\mu},~~\text{ for all } t\geq 0. 
\]
We further define the spaces $\cH^\alpha=D((1-A)^\alpha)$ for $\alpha\geq 0$ endowed with the norm $\|\cdot\|_\alpha=\|(1-A)^\alpha\cdot \|$ and scalar product $\langle u, v \rangle_\alpha=\langle (1-A)^\alpha u,(1-A)^\alpha v \rangle$ and set $\cH^{-\alpha}=(\cH^\alpha)^*$ the dual of $\cH^\alpha$. It is well-known that $(e^{tA})_{t\geq 0}$ is an analytic semigroup on $\cH^\alpha$ for every $\alpha\in\mathbb{R}$. Finally, we have that $\cN\subset \cH^\alpha$ for all $\alpha>0$ since  $(1-A)^\alpha\cN=\cN$. 
\end{assumptions}

Under our assumptions we have for some constant $C>0$ depending on $\alpha>0$ that
$\|A^\alpha P_su\| \geq C\|P_su\|$ 
for all $u\in\cH$, which we use frequently.

\begin{assumptions}\label{f}(Nonlinearity)
We assume that there exists a Banach space $X$ such that 
\[
\cH^{\alpha}  \subset X \subset \cH
\]
for $\alpha\in(0,1/2)$ with continuous and dense embeddings.
Moreover, the mapping $\mathcal{F}: X\to X^*\subset \cH^{-\alpha}$
is a stable cubic (i.e.\ trilinear) nonlinearity with
\begin{equation}
\label{e:stableF}
\langle \cF(u)-\cF(v), u-v\rangle 
\leq - c\|u-v\|_X^4  ,\quad \mbox{ for } u,v \in X.
\end{equation} 
\end{assumptions}
Let us remark that we can allow terms like $C\|u-v\|^2$ on the r.h.s.\ of \eqref{e:stableF}, but we can always modify the linear term to remove these terms.

In Section~\ref{seven}, we give several examples that will fit in our abstract setting like Allen-Cahn, Swift-Hohenberg, and a surface growth model.
\begin{assumptions}\label{s}(Stochastic convolution)
We assume that the stochastic convolution
\[ Z(t)=\int_0^t e^{A(t-s)}~\txtd W_s\]
is well-defined and has continuous trajectories in $X$. 
\end{assumptions}

Let us remark that for a trace-class Wiener $W$ process on $\cH$, this assumption is automatically satisfied, since $Z\in C([0,T];\cH^\alpha)$ for $\alpha<1/2$ by the factorization method~\cite{DaPratoZ}. 

\begin{remark}
For $\cF$ we have the following sign condition. For any positive $\delta>0$ there is a constant $C>0$ depending on $\delta$ such that for all $u, z\in X$
\begin{equation}
\label{e:boundF}
\langle\cF(u+z),u\rangle 
\leq -c\|u+z\|^4_X + C\|u+z\|^3_X\|z\| 
\leq - \delta\|u\|^4_X + C_\delta \|z\|_X^4.
\end{equation}

As $\cF$ is trilinear we have that $\cF$ is Fr\'echet-differentiable with 
\[
D\cF(u)[h]= \cF(u,u,h)+\cF(u,h,u)+\cF(h,u,u).
\] 
Moreover, for $u,h\in  X$ we obtain due to~\eqref{e:boundF}
\begin{equation}
    \label{e:sign}
    \langle D\cF(u)h,h\rangle 
= \lim_{t\to 0} \frac1t \langle \cF(u+th)-\cF(u), h\rangle
\leq - \lim_{t\to 0} \frac1{t^2} \|th\|^4
= 0 .
\end{equation}
\end{remark}

\begin{definition}\label{order}
For the $\cO$-notation here we use moments uniformly in time in the space $X$.
This means that a process $M$ is $\cO(f)$ for a term $f$ on an interval $I$ in the space $X$, if for all $p>1$ there is a constant $C_p>0$ such that $\E\sup_{t\in I}\|M(t)\|^p_X \leq C_p f$. 
For time independent quantities we use the similar notation without the supremum in time. 
If the process $M$ and the bound $f$ depends on some small quantity  $\varepsilon>0$ we assume that the constant $C_p$ is independent of $\varepsilon \in (0,\varepsilon_0]$ for some $\varepsilon_0>0$. 
\end{definition}

\begin{assumptions}\label{o:z}
For the stochastic convolution, 
we have for every small $\kappa>0$
\begin{align}\label{scaling:z}P_s Z = \cO(T^{\kappa})\qquad\text{and}\qquad 
P_c Z = P_c W = \cO( T^{1/2})
\end{align}
on any interval $[0,T_0]$ in the space $X$. 
\end{assumptions}
\begin{remark}
These bounds are natural and can be checked in applications. 
See for example~\cite{Bl:07}, where the bound on $P_s$ is obtained by the factorization method~\cite{DaPratoZ}. 
We also remark that the bound by $T^\kappa$ is often not optimal, but sufficient for our applications. 
A sharp bound should have a logarithmic term~\cite{Q}.
\end{remark}


\subsection{Existence of solutions} 
 For SPDEs of our type the existence and uniqueness of solutions is in many examples well known. 
 To verify this there are several options. 
 
 Firstly,
 by fixed-point arguments one shows the local existence of unique solutions in the space $X$. One only needs some results on the semigroup acting on the space $X$, which follows from our assumptions. 
In a similar way
one can use the fact that $\mathcal{F}: \cH^{\alpha} \to  \cH^{-\alpha} $
 and establish local well-posedness in $\cH^{\alpha}$, but only if the stochastic convolution is in that space, too. 
 
 Alternatively,  we can rely on the Galerkin method and the standard transformation $w=u-Z$ that solves the random PDE
 \[
\partial_t w = Aw +\nu w +\cF(w+Z)
 \]
 and we can proceed with classical pathwise existence result.
 See for example \cite{Te:97,Rou:13}.
 
This is based on \eqref{e:boundF}
giving regularity in $L^4(0,T,X)$, together with the compact embedding of $X$ into $\cH^{1/2}$ and Aubin-Lions Lemma.

For initial conditions in $\cH$ and $Z$ being a continuous stochastic process with values in $X$ this shows global 
existence of solutions such that 
\[
u-Z \in L^2(0,T,\cH^{1/2}) \cap  C^0([0,T],\cH)\cap L^4(0,T,X) 
\]
which also implies some regularity of  $\partial_t(u-Z)$ as 
$A(u-Z) \in L^{2}(0,T,\cH^{-1/2})$ and 
$\cF(u) \in L^{4/3}(0,T,X^*)$.

The pathwise uniqueness of solutions follows immediately from \eqref{e:stableF}. For the 
difference $d=u_1-u_2$ of two solutions 
$u_1$ and $u_2$ satisfying
\[
\partial_t d = Ad +\nu d +\cF(u_1)-\cF(u_2)
\]
we only need 
the differentiability of the $\cH$-norm to conclude
\[\partial_t \|d\|^2 = \langle Ad + \nu d +\cF(u_1)-\cF(u_2), d \rangle \leq \nu \|d\|^2.
\]
The differentiability of the norm follows, as  we have \[d=(u_1-Z)-(u_2-Z) \in L^2(0,T,\cH^{1/2}) \cap  L^\infty(0,T,\cH)
\]
by standard parabolic regularity together with $d \in L^4(0,T,X)$ and 
$\cF(u_i) \in L^{4/3}(0,T,X^*)$.

With the arguments sketched above one can prove the following theorem, which we state without proof.
\begin{theorem}
\label{thm:exSPDE}
Let Assumptions~\ref{a},~\ref{f},~\ref{s} be satisfied.
Then for all initial conditions $u_0\in\cH$ there is a unique (up to global null sets)
stochastic process $u$ with continuous paths in $\cH$, 
which is a weak solution of~\eqref{spde}
and satisfies for all $T>0$ 
\[
u-Z \in L^2(0,T,\cH^{1/2}) \cap  C^0([0,T],\cH) 
\cap L^4(0,T,X). 
\]
\end{theorem}

Next, we recall some basic facts about  random dynamical system (RDS) without giving too many technical details.
Under our assumptions, we know due to standard results~\cite{C,FlGeSch:17} that the solution operator of~\eqref{spde1} generates a RDS, which is a family of linear maps $\varphi:\R_+\times\Omega\times\cH\to \cH$ satisfying the cocycle property, i.e.\
\[\varphi(t+s,\omega,u_0) =\varphi(t,\theta_s\omega,\varphi(s,\omega,u_0) \text{ for all } s,t\in\R_{+}. \]
Here $\theta_t\omega(s):=\omega(t+s)-\omega(s)$ for $\omega\in\Omega$ and $t,s\in\R$ denotes the Wiener shift, where $(\Omega,\cF,\P)$ stands for the canonical probability space associated to the two-sided Wiener process $(W(t))_{t\in[0,T]}$ and $\omega_t=W_t(\omega)$.
Furthermore the random dynamical system $\varphi$ has a random attractor~\cite{FlGeSch:17, MaSc:04} which is a singleton, i.e.\ there exists a random variable $a:\Omega\to \cH$ such that
\[ \varphi(t,\omega,a(\omega))=a(\theta_t\omega) .\]
In this framework we denote the two-sided filtration on $(\Omega,\cF)$ by $\cF_s^t:=\sigma(W_u-W_v, s\leq u, v\leq t)$ for all $s\leq t$ with $\theta^{-1}_{\tau}\cF^t_s=\cF^{t+\tau}_{s+\tau}$.


\subsection{Finite-time Lyapunov-exponents}

Consider the RDS generated by the solutions $u$ of \eqref{spde}.
The linearization $D_{u_0(\omega)}\varphi(t,\omega,u_0(\omega))=D_{u_0(\omega)} u(t,\omega,u_0(\omega))$ of~\eqref{spde1} around a solution $u(t,\omega,u_0(\omega))$ with initial condition $u_0$ is defined as the solution of  the linear PDE called also the first variation equation, see~\cite{BlEnNe:21}
\begin{align}\label{linearization}
\begin{cases}
    \txtd v =  [Av + \nu v + D \cF(u) v]~\txtd t \\
    v(0)=v_0.
    \end{cases}
\end{align}
\begin{remark}
The Fr\'echet differentiablity of the cocycle follows regarding that $u\in L^2(0,T;\cH^{1/2})$ due to~\cite[Lemma 4.4]{deb}.
\end{remark}
For $t>0$ we denote the random solution operator $U_{u_0}(t):\cH\to\cH$ such that $v(t)=U_{u_0}(t)v_0$, where $v$ is a solution of \eqref{linearization} given the initial condition $v_0\in \cH$.

\begin{remark}
Note that for any solution $u\in L^4(0,T,X)$ we have $\cF(u)\in L^{4/3}(0,T,X^*) \subset L^{4/3}(0,T,{\cH}^{-\alpha})$.
We can now use pathwise deterministic theory for linear PDEs. 
For example Galerkin methods show that for given $v_0\in\cH$ there is an (up to global null sets) unique stochastic process $v$ with continuous paths in $\cH$ and $v \in L^2(0,T;\cH^{1/2})$ for all $T>0$ that solves \eqref{linearization}. 
\end{remark}

We define the finite-time Lyapunov exponent as in~\cite{BlEnNe:21}.

\begin{definition}{\em (Finite-time Lyapunov exponent)}. Let $t>0$ be fixed. We call a finite-time Lyapunov exponent for a solution $u$ of the SPDE with (random) initial condition $u_0$ 
\begin{equation}\label{ftle}
    \lambda_t(u_0):=\lambda(t,\omega,u_0(\omega)) =\frac{1}{t} \ln \left( \| U_{u_0}(t) \|_{\cL(\cH)}\right).
\end{equation}
\end{definition}

From the definition it is clear that finite-time Lyaponov exponents measure local expansion rates of nearby solutions. Negative finite-time Lyapunov exponents indicate attraction whereas positive ones indicate that nearby solutions tend to separate on a finite time horizon.

\begin{remark}
\label{propLyap}
We can compute $\|U_{u_0}\|_{\cL(\cH)}$ as follows
\begin{eqnarray*}
\| U_{u_0}(t) \|_{L(\cH)}
&=& \sup\{ \|v(t)\| / \|v(0)\| \ : \  v \text{ solves \eqref{linearization} with }v(0)\not=0 \} \\
&=& \sup\{ \|v(t)\|  \ : \  v \text{ solves \eqref{linearization} with }\|v(0)\|=1 \}.
\end{eqnarray*}
\end{remark} 

\begin{definition}\label{as:lyexp}
{\em(Asymptotic Lyapunov-exponent)} The asymptotic Lyapunov exponent around $u_0\in\cH$ is defined as $$ \lambda(u_0):=\lim\limits_{t\to\infty} \frac{1}{t} \ln (\| U_{u_0}(t)\|_{\cL(\cH)} ). $$
\end{definition}
\begin{remark}
In contrast to~\eqref{ftle}, due to the ergodic theorem the limit in Definition~\ref{as:lyexp} exists and is deterministic (independent of $\omega$) with probability one. 
Here we need to rely on the ergodicity of a stationary solution, which for instance holds if the random attractor is a single point.  
Estimating this quantity for SPDEs is a highly challenging task, compare~\cite{GeTs:22,BeBlPS:22}.
\end{remark}


\section{Stability for negative $\nu$}\label{sec:stability}

This is the trivial case, which is well known in many examples. See for example~\cite{MaSc:04}.
\begin{lemma}
The SPDE~\eqref{spde} has a global attractor which is a unique random fixed point.
\end{lemma}
\begin{proof}
We only sketch the standard arguments, see~\cite{MaSc:04} for further details. We consider the difference of two solutions $d=u_1-u_2\in L^2(0,T;\cH^{1/2})\cap L^\infty(0,T;\cH)$ which satisfies the equation
\[
\partial_t d = Ad+\nu d +\cF(u_1)-\cF(u_2). 
\]
Using standard energy estimates we obtain using the assumptions on $A$ and~\eqref{e:stableF}
\begin{align*}
    \frac{1}{2}\partial_t\|d\|^2&=\langle Ad, d\rangle +\nu \|d\|^2 +\langle F(u_1)-F(u_2), u_1-u_2\rangle\\
    & \leq \nu \|d\|^2 -c \|d\|^4_X,
\end{align*}
which implies that $\|d(t)\|\leq \|d(0)\| e^{t\nu}$ for all $t>0$. 
Since $\nu<0$ this means that all  solutions converge towards each other yielding the existence of a singleton attractor. 
\qed
\end{proof}

As a consequence of the previous result we have that all finite time Lyapunov exponents are negative. 

\begin{theorem}\label{l:neg}
Let Assumptions~\ref{a},\ref{f},\ref{s} hold true and let $\nu<0$. Furthermore let $u$ be a solution of~\eqref{spde} in the sense of Theorem \ref{thm:exSPDE} with random initial condition $u_0\in\cH$.
Then for all $T>0$ we have with probability one
\[
\mathbb{P}(\lambda_T(u_0)\leq \nu) =1.
\]
\end{theorem}
\begin{proof}
The proof is similar to~\cite[Proposition 3.1 a)]{BlEnNe:21}. We consider a solution $v$ of the linearized problem~\eqref{linearization} around a solution $u$ of~\eqref{spde} with random $\cH$-valued initial condition $u_0$. 
Recalling that $v\in H^1(0,T,\cH^{-1/2})\cap L^2(0,T,\cH^{1/2})\cap C(0,T,\cH)$ we obtain using \eqref{e:sign} the standard energy estimate
\begin{align*}
    \frac{1}{2}\partial_t \|v\|^2& =\langle A v, v\rangle +\nu \|v\|^2 +\langle D\cF(u)v, v\rangle 
    \leq \nu \|v\|^2.
\end{align*}
This implies that $ \|v(t)\| \leq \| v(0)\|e^{t\nu}$ for all $t>0$. Due to Remark \ref{propLyap}
we have for any time $T>0$
\[
\lambda_T(u_0) =\frac{1}{T}\ln (\|U_{u_0}(T)\|_{L(\cH)} )\leq  \nu 
\]
which finishes the proof.
\qed
\end{proof}


\begin{remark}
Note that the result is still true for $\nu>0$ but there it only gives a positive upper bound for the Lyapunov exponents, which is not relevant.  
\end{remark}

\section{Lyapunov Exponents for  Amplitude Equations}\label{ly:a}

Here we summarize properties of the one-dimensional 
SDE given by the amplitude equation. Let us first start with an important technical lemma about the projected nonlinearity.
\begin{lemma}
The nonlinearity 
\[\mathcal{F}_c(b) :=P_c \mathcal{F}(b) \quad \text{for } b\in\cN \]
is always a stable cubic as a map on $\cN\approx \R$.
\end{lemma}
\begin{proof}
Obviously, since $\cN$ is one-dimensional we get that $\cN=\text{span}\{e\}$ for a fixed $e\in\cN$ with $\|e\|=1$. Note that $e$ belongs to the kernel of $A$ and thus $e\in \cH^\alpha$ for all $\alpha>0$, meaning that $e\in X$, due to the embedding $\cH^\alpha\subset X$.
Therefore, writing $b=\xi\cdot e$ and using~\eqref{e:boundF} we get  \[\langle\mathcal{F}_c(b),e\rangle=
 \xi^3\underbrace{\langle\mathcal{F}_c(e),e\rangle}_{\leq-c\|e\|^4_X < 0}.
\]
Thus, if we identify $\cN\approx \R$ the cubic nonlinearity is always a stable cubic $ - C \xi^3$ for a positive constant. Analogously we have for the derivative of $\cF_c$ using~\eqref{e:boundF}
\[\langle D \cF_c(b)e, e\rangle = 3\xi^2 \langle D\cF_c(e)e, e\rangle \leq 0
\]
where equality only holds if $b=0$.
\qed
\end{proof}

\newcommand{\rev}[1]{\textcolor{blue}{#1}}

For the approximation via amplitude equations later, we assume that $\nu\geq 0$ and fix a small quantity $\varepsilon>0$, where we always assume the upper bounds
\[
\nu=\cO(\varepsilon^2) 
\text{ and }
\sigma=\cO(\varepsilon^2). 
\]

Using the nonlinearity $\mathcal{F}_c$ and depending on the order of $\nu$ and $\sigma$ we obtain in Section~\ref{six} 
amplitude equations of two different types. 
Mostly, we have 
 \begin{align}\label{a1}
\txtd b=(b+\cF_c(b))~\txtd T + (\sigma/\nu) ~\txtd \beta_{\sqrt{\nu}}(T)
\end{align}
and for $\nu\ll\sigma$   
\begin{align}\label{a2}
\txtd b=\cF_c(b)~\txtd T+\txtd\beta_{\sqrt{\sigma}}(T),
\end{align}
where $\beta_{c}$ is an $\cN$-valued Brownian motion that is rescaled in time by $\gamma$, i.e.\ $\beta_\gamma(T)=\gamma\beta(T/\gamma^2)$.

By standard theory~\cite{CrFl:98}, due to monotonicity, these SDEs have a random attractor $\{a(\omega)\}_{\omega\in\Omega}$
which is a single fixed point. Our aim is to analyze the finite-time Lyapunov exponents around this random fixed point.

 \begin{remark}\label{fp}
 We know that $\E \|a\|^p_{\cN}<\infty$ for all $p\geq 1$, which means that $a=\cO(1)$. More precisely, using the Fokker-Planck equation one can show that the Markov semigroup associated to~\eqref{a1} has an invariant measure with density (~\cite[p.~474]{Ar:98})
 \[ p(x)=\frac{1}{N} \exp \Big(\frac{\nu^2}{\sigma^2} \Big(x^2 - \frac{1}{2}x^4 \Big) \Big), \]
 where $N$ is a normalization constant. Analogously, the density for~\eqref{a2} reads as
 \[  p(x)=\frac{1}{N} \exp \Big( -\frac{1}{2}x^4\Big) .\] 
 Moreover using that $\{a(\omega)\}_{\omega\in\Omega}$ is the random fixed point of~\eqref{a1} respectively~\eqref{a2}, we can infer that $a(\theta_T\omega)=\cO(1)$ for $T\in[0,T_0]$, see also the computations below.
 \end{remark}

We now provide results for the finite-time Lyapunov exponents of these SDEs around the random attractor $a(\omega)$.

 \begin{lemma}\label{l1}{\em(Positive finite-time Lyapunov-exponent for~\eqref{a1})} There exists a constant $c>0$ independent of the choice of the Brownian motion and a set $\Omega_0\in \cF^{T}_{-\infty}$ with $\mathbb{P}(\Omega_0)\geq c>0$ such that $\lambda_T(a(\omega))\geq \frac14$. 
 \end{lemma}

 
 
 \begin{proof}
 Recalling that $a(\omega)$ is the random attractor of~\eqref{a1}, we define 
 $$
 A_1:=\{\omega\in\Omega: a(\omega)\in(-\eta,\eta)\} \in \cF^0_{-\infty},
 $$ 
 where $\eta:=\frac{\delta}{2(1+\frac{\sigma}{\nu}) }$, with $\delta>0$ to be chosen later
and $$A_2:=\{ \omega\in\Omega : \sup\limits_{t\in[0,T]} |\beta_{\sqrt{\nu}}(t)|\leq \frac{\eta}{2} \}\in \cF^T_0.$$
We know that $\P(A_1)>0$ due to the fact that $a(\omega)$ is distributed according to an invariant measure which is supported on the real line. Furthermore, $\P(A_2)>0$ since in our case $\varepsilon=\sqrt{\nu}$, so $\beta_{\sqrt{\nu}}(T)=\varepsilon\beta(T\varepsilon^{-2})$ is just a rescaled Brownian motion.
We define  $\Omega_0:=A_1\cap A_2\in \cF^{T}_{-\infty}$.

Since $A_1$ and $A_2$ are independent,
in order to give a uniform lower bound for the probability of $\Omega_0$, we can consider the set separately. First for $A_2$ we apply Doob's inequality for every $p\geq 1$, which yields
\[
\P \Big(\sup\limits_{t\in[0,T]} |\beta_{\sqrt{\nu}}(t)|>\frac{\eta}{2}\Big) \leq C_p  \Big(\frac{2}{\eta}\Big)^p T, 
\]
for some constant $C_p>0$. For $A_1$ we can simply rely on the Fokker-Planck equation for the law of $a$ to get a uniform lower bound, recall Remark~\ref{fp}.

Let us now turn to the lower bound on the Lyapunov exponent.
Using the equation~\eqref{a1} and regarding that $\cF_c$ is a stable cubic, one can verify, similar to~\cite{R}, that for $\omega\in\Omega_0$
\[ |a(\theta_T\omega)| \leq (1+\frac{\sigma}{\nu}  ) \eta e^T<\delta \text{ for all } T\in[0,T_0].
\]
This further implies that the finite-time Lyapunov exponent of~\eqref{a1} is positive with positive probability, since for $\omega\in\Omega_0$ we have
\[  \lambda_T(a(\omega))  =\frac{1}{T} \ln\Big( \exp \Big(T +\int_0^T D\cF_c(a(\theta_t\omega))~\txtd t\Big)\Big) \geq \frac{1}{4}, \]
choosing $\delta:=\frac{1}{2}$.
\qed
\end{proof}

Now let us turn to the negative finite-time Lyapunov-exponent for~\eqref{a2}). Here we need a quantitative result of Birkhoff type.

\begin{lemma}\label{birkhoff}
  For all probabilities $p\in(0,1)$ there exists a time $T^*>0$ and a set $\tilde{\Omega}_p$ with $\P(\tilde{\Omega}_p)\geq1-p$, such that for all $T\geq T^*$ and $\omega\in\tilde{\Omega}_p$ it holds that
    \[\frac1T \int_0^T a(\theta_s\omega)^2 ~\txtd s \geq \frac14 \E a(\omega)^2. \] 
\end{lemma}

  We can apply this result in our setting since we will be interested in time scales of order between $1/\sqrt{\nu}$ and $1/\nu$.

\begin{proof}
We split the proof into the following steps.

\paragraph{First step:}  Using Birkhoff's ergodic theorem
     we obtain
     \[ \lim\limits_{T\to\infty} \frac{1}{T} \int_0^T a(\theta_s\omega)^2~\txtd s =\E a(\omega)^2, \]
which means that
\[ \Big| \frac{1}{T} \int_0^T a(\theta_s\omega)^2~\txtd s -  \E a(\omega)^2 \Big| \to 0 \text { as } T\to\infty .\]
This statement is not enough for our aims, since the quantity above can only be made arbitrarily small for all $T\geq T^*(\omega)$, i.e.\ after a random time. But we need precise quantitative bounds on the time scale and of course this shall not depend on $\omega$.

     
\paragraph{Second Step:} We can show that for even integers $k$ we have  uniformly in $T$ 
\[ \E \Big|\frac1T \int_0^T [a^2(\theta_s\omega)- \E a(\omega)^2] ~\txtd s \Big|^{k}  \leq \E |[a^2(\omega)-\E a(\omega)^2]|^k. 
\]
This follows using H\"older's inequality combined with the fact that the random attractor $\{a(\omega)\}_{\omega\in\Omega}$ is the unique stationary solution of~\eqref{a2}. In particular this means that 
$\E|a(\theta_s\omega)|^{2k}=\E |a(\omega)|^{2k}$ 
for all $s\in[0,T]$. 
More precisely, H\"older's inequality gives us for $\ell\in\N$
\begin{align*} \E\Big|\int_0^T [a(\theta_s\omega)^2 -\E a(\omega)^2]~\txtd s \Big|^{2\ell} 
&\leq \int_0^T \E| [a(\theta_s\omega)^2 -\E a(\omega)^2] |^{2\ell}~\txtd s~ T^{2\ell-1} \\
& = T^{2\ell} \E  |[a(\omega)^2 -\E a(\omega)^2]|^{2\ell},\qquad s\in[0,T],
\end{align*}
due to the stationarity of $a$. Dividing by $T^{2\ell}$ proves the statement.

\paragraph{Third Step:} The convergence in Birkhoff's ergodic theorem holds in $L^p$ for all $p<k$. The previous moment bounds imply the uniform integrability of 
\[
X_T: =\frac{1}{T}\int_0^T a(\theta_s\omega)^2~\txtd s - \E a(\omega)^2.\]
We have thus for all $C>0$
\[
\E |X_T|^{p} = \E \chi_{\{|X_T|\leq C \}}|X_T|^{p} +  \E \chi_{\{|X_T|> C \}}|X_T|^{p} .
\]
The first term obviously converges to $0$ by dominated convergence, 
the second term is bounded by 
$C^{-p}\E X_T^{2p}$  
using Cauchy-Schwarz and Chebychev's inequality. 
Thus by the second step we have a bound of order $\cO(C^{-p})$ uniformly in $T$. Hence, for all $p>1$ we have 
$\E |X_T|^{p} \to 0$ for $T\to\infty$.

\paragraph{Fourth step:} In conclusion we obtain that for all  $\tilde{\delta}>0$ that 
\[ \mathbb{P} \Big(  \Big|\frac1T \int_0^T a(\theta_s\omega)^2 ~\txtd s  - \E a(\omega)^2 \Big| \leq \tilde{\delta}  \Big)  \to 1 \quad\text{ for } T\to \infty. 
\]
For $T\geq T^*(p)$ and $\tilde\delta:= \frac34 \E a(\omega)^2$ 
we find a set $\tilde\Omega_p$ with measure larger than $1-p$, such that 
\[\frac1T \int_0^T a(\theta_s\omega)^2 ~\txtd s \geq \frac14 \E a(\omega)^2. \] 
\qed
\end{proof}

\begin{lemma}\label{l2}{ \em(Negative finite-time Lyapunov-exponent for~\eqref{a2})} The finite-time Lyapunov-exponents around the attractor of~\eqref{a2} are negative with positive probability.
\end{lemma}
\begin{proof}
The linearization of~\eqref{a2} around the random atractor $\{a(\omega)\}_{\omega\in\Omega}$ entails in this case
\[ \txtd b =D\cF_c(a)~\txtd t. \]
Due to~\eqref{e:sign} and  Lemma~\ref{birkhoff} we obtain the existence of a time $T^*>0$, of a set $\Omega_{T^*}$ with probability almost one and of a constant $c>0$ such that for all $T\geq T^*$
\[
\lambda_T(a(\omega)) =\frac{1}{T} \ln \exp \Big( \int_0^T D\cF_c(a(\theta_t\omega )~\txtd t\Big) < -c<0~\quad \text{ for }\omega\in\Omega_{T^*}.
\]
\qed
\end{proof}

\section{Approximation of the SPDE via Amplitude Equations}\label{five}

We recall here briefly the steps of the approximation of~\eqref{spde1} via amplitude equations. Note that we work in a slightly different setting than the results in the literature~\cite{BlHa:04,Bl:05,Bl:07}. 

We consider $\nu\ \geq 0 $ and fix $\varepsilon\in(0,\varepsilon_0]$ for some
$\varepsilon_0 >0$ sufficiently small. 
\begin{assumptions}\label{ordnung} 
In our setting we make the following assumptions
\[
\nu =\cO(\varepsilon^2) \text{ and } \sigma=\cO(\varepsilon^2).
\]
\end{assumptions}
Recall that this does not exclude $\nu$ or $\sigma$ to be much smaller that $\varepsilon^2$.

\begin{ansatz}
  The process $b$ is an $\cN$-valued process and solves the amplitude equation 
\begin{equation}
    \label{e:AE}
    \txtd b =( \nu\varepsilon^{-2}  b + \mathcal{F}_c(b) )~\txtd T+ \sigma\varepsilon^{-2} \txtd \beta_\varepsilon(T).
\end{equation}  
Recall that $\beta_\varepsilon$ is a rescaled Brownian motion defined by 
$\beta_\varepsilon(T)=\varepsilon \beta(T\varepsilon^{-2})$ where $\beta=P_cW$.
\end{ansatz}




In the following, given our $\varepsilon>0$ we fix a time $T_\varepsilon=T_0 \varepsilon^{-2}$ and consider $T_0$ as a constant.

\begin{theorem}\label{thm:aprox}
Let $u$ be a solution of the SPDE~\eqref{spde} with initial condition $u_0=\cO(\varepsilon)$ in $\cH$
such that $P_su_0=\cO(\varepsilon^2)$  in $\cH$. Further, let $b$ be a solution of\eqref{e:AE} with $b(0)- \varepsilon^{-1} u_0 = \cO(\varepsilon)$. 
Then 
\[u-\varepsilon b(\varepsilon^2\cdot) = \cO(\varepsilon^{2-\kappa})
\quad\text{ on }[0,T_\varepsilon]\text{ in }\cH \text{ for all small } \kappa>0. 
\]
\end{theorem}

We split the proof into several steps in order to bound the error 
$u-\varepsilon b(\varepsilon^2\cdot) = P_su + (P_c u-\varepsilon b(\varepsilon^2\cdot)).$

\paragraph{1st step:} If $u_0=\cO(\varepsilon)$ in $\cH$, 
then we show that $u=\cO(\varepsilon)$ in $\cH$ on $[0,T_\varepsilon]$. 
Using the standard transformation
$\tilde{u}:=u - \sigma Z$ we obtain the partial differential equation with random coefficients
\[
\partial_t \tilde{u} 
= A\tilde{u} + \nu (\tilde{u}+ \sigma Z) + \cF(\tilde{u}+\sigma Z). 
\]
Note that $\sigma Z=\cO(\varepsilon)$ on $[0,T_\varepsilon]$ due to~\eqref{scaling:z}.
Thus it is enough to bound $\tilde{u}$.
Due to~\eqref{e:boundF} using Young's inequality we obtain the estimate 
\begin{eqnarray*}
\frac12\partial_t \|\tilde{u}\|^2 
&\leq & \nu \langle \tilde{u}+\sigma Z , \tilde{u}\rangle + \langle \cF(\tilde{u}+\sigma Z) , \tilde{u}\rangle  \\ 
&\leq & \nu \|\tilde{u}\|^2+ \nu\sigma \langle Z,\tilde{u}\rangle + C\sigma^4\|Z\|^4_X - \delta \|\tilde{u}\|^4_X \\
&\leq& \frac12(\nu\|\tilde{u}\|^2 -\delta \|\tilde{u}\|^4_X )
+ C(\nu\sigma^2 \| Z \|^2  +\sigma^4\|Z\|^4_X) \\
&\leq&  C(\nu^2+\nu\sigma^2 \| Z \|^2  +\sigma^4\|Z\|^4_X) = \cO(\varepsilon^4).
\end{eqnarray*}
Here we used again that $\sigma Z=\cO(\varepsilon)$ on $[0,T_\varepsilon]$, where $T_\varepsilon=\cO(\varepsilon^{-2})$ and that $\nu=\cO(\varepsilon^2)$. 
This completes the proof of the first step.
Additionally, we can also conclude that 
\[
\frac12\partial_t \|\tilde{u}\|^2  = - \frac14\delta \|\tilde{u}\|^4_X +  \cO(\varepsilon^4),
\]
which gives the $L^4(0,T_\varepsilon,X)$ bound on $\tilde{u}$
\begin{equation}
\label{e:bound-e-int}
   \int_0^{T_\varepsilon} \|\tilde{u}(t)\|^4_X~\txtd t =   \cO(\varepsilon^2).
\end{equation}

\paragraph{2nd step:}  If $u_0=\cO(\varepsilon)$ and $P_su_0=\cO(\varepsilon^{2})$ in $\cH$, then we show that $P_s u=\cO(\varepsilon^{2-\kappa})$ in $\cH$ on $[0,T_\varepsilon]$. To this aim we use the splitting
\[
u=  P_cu+P_s u:= u_c+ u_s,
\] 
and define with $Z_s:=P_s Z$. Again we use the standard transformation
\[
\tilde{u}_s=u_s-\sigma Z_s = P_s \tilde{u}. 
\] 
Using that $\tilde{u}=u - \sigma Z$ and taking the stable projection entails
\[
\partial_t \tilde{u}_s
= A \tilde{u}_s+\nu (\tilde{u}_s+\sigma Z_s)  +P_s\cF(u_c+\sigma Z_s+\tilde{u}_s).
\]
The Assumption~\ref{a} implies that the quadratic form of $A$ on $P_s\cH$ is bounded from below by a positive constant. 
Therefore we further obtain 
\[
\frac12 \partial_t \|\tilde{u}_s\|^2
\leq  -c\| \tilde{u}_s\|^2 + \nu \|\tilde{u}_s\|^2
+ \nu \langle  \tilde{u}_s,\sigma Z_s\rangle + \langle \cF(u_c+\sigma Z_s+\tilde{u}_s),\tilde{u}_s \rangle.
\]
Using~\eqref{e:boundF} together with the fact that $\varepsilon_0$ is sufficiently small and thus $\nu=\cO(\varepsilon^2)$ is small, we derive the energy estimate

\begin{equation}
\label{e:PsuinX}
    \frac12 \partial_t \|\tilde{u}_s\|^2
\leq  -\frac{c}2 \| \tilde{u}_s\|^2 + C\nu \sigma^2 \|Z_s\|^2
 + C(\|u_c\|^4_X+ \sigma^4\|Z_s\|^4_X) 
-\delta \|\tilde{u}_s\|^4_X,
\end{equation}
for two universal constants $c,C>0$. 
Hence, via a Gronwall type estimate, for all $t\leq T_\varepsilon$
\[
\|\tilde{u}_s(t)\|^2 
\leq \|\tilde{u}_s(0)\|^2 
+ C\int_0^{T_\varepsilon} e^{-c(t-\tau)}
\Big( \nu \sigma^2 \|Z_s\|^2
 + \|u_c\|^4_X+ \sigma^4\|Z_s\|^4_X\Big) ~\txtd \tau .
\]
We use that all norms are equivalent on $\cN$ together with the bounds $\sigma =\cO(\varepsilon^2)$, $\nu=\cO(\varepsilon^2)$, $\|Z_s\|_X=\cO(T^\kappa)$ and $P_su_0=\cO(\varepsilon^2)$
to obtain 
\[
\|\tilde{u}_s\|^2 =\cO(\varepsilon^4) \text{ on }[0,T_\varepsilon].
\]
Thus 
\[
\|{u}_s\| 
\leq \|\tilde{u}_s\| + \sigma\|Z_s\| 
=\cO(\varepsilon^2) + \cO(\varepsilon^{2-\kappa})\text{ on }[0,T_\varepsilon]
\]
which bounds the error on $P_s\cH$.
\paragraph{3rd step:} Bounds on $P_s u$ in $X$. 

Note that from \eqref{e:PsuinX} we also obtain 
\[
\delta \int_0^t\|\tilde{u}_s\|^4_X~\txtd t
\leq  \|\tilde{u}_s(0) \|^2 + C \int_0^t (\nu \sigma^2 \|Z_s\|^2
 + \|u_c\|^4_X+ \|Z_s\|^4_X)~\txtd  t
\]
and thus 
\[
\int_0^{T_\varepsilon}\|\tilde{u}_s(t)\|^4_X ~\txtd t = \cO(\varepsilon^4)
\]
or 
\[
\int_0^{T_0}\|\tilde{u}_s(\varepsilon^{-2}t )\|^4_X ~\txtd t = \cO(\varepsilon^6).
\]

\paragraph{4th step:} The order of the amplitude equation is $b=\cO(1)$ on $[0,T_0]$. The proof is fairly standard, compare~\cite[Lemma 4.3]{BlHa:04}. For completeness of presentation, we only sketch the main ideas. Subtracting the scalar stochastic convolution 
\begin{align*}
    z(T):=\sigma \varepsilon^{-2} \int\limits_0^T e^{-(\nu\varepsilon^{-2})(T-S)}~\txtd \beta_\varepsilon(S),
\end{align*}
we obtain the random ODE on the slow-time scale $T>0$ 
\begin{align*}
\partial_T \tilde{b} = \nu \varepsilon^{-2} (\tilde{b}+z) + \cF_c(\tilde{b}+z) + \nu \varepsilon^{-2} z,
\end{align*}
where we set $\tilde{b}:=b-z$.
Now taking the inner product with $\tilde{b}$ and using~\eqref{e:sign} to bound the stable cubic nonlinearity $\cF_c$, we obtain 
\begin{align*}
\frac{1}{2}\partial_T \|\tilde{b}\|^2 &=\langle \nu \varepsilon^{-2}(\tilde{b}+z),\tilde{b}\rangle + \langle \cF_c(\tilde{b}+z),\tilde{b}\rangle +\langle \nu \varepsilon^{-2} z, \tilde{b}\rangle\nonumber\\
& \leq K \nu\varepsilon^{-2} \|\tilde{b}\|^2 + K \nu \varepsilon^{-2} \|z\|^2 + \langle\cF_c (\tilde{b}+z),\tilde{b} \rangle\nonumber\\
& \leq K \nu \varepsilon^{-2}\|\tilde{b}\|^2 + K \nu\varepsilon^{-2} \|z\|^2 + C_\delta \|z\|^4 - \delta \|\tilde{b}\|^4\nonumber\\
& \leq K \nu \varepsilon^{-2} \|\tilde{b}\|^2 + K \nu \varepsilon^{-2}(1+\|z\|^2)^2,
\end{align*}
for positive constants all denoted by $K$. 
Now, Gronwall's inequality entails for $T\in[0,T_0]$
\begin{align*}
    \|\tilde{b}(T)\|^2 \leq e ^{\nu \varepsilon^{-2}K T} \|\tilde{b}(0)\|^2 + K\nu \varepsilon^{-2}\int_0^T e ^{\nu\varepsilon^{-2}K(T-\tau)}(1+\|z(\tau)\|^2)^2~\txtd \tau.
\end{align*}
This proves the statement regarding that $\nu=\cO(\varepsilon^{2})$. 

\paragraph{5th step:} We show that $\varepsilon^{-1}P_cu(\varepsilon^{-2} \cdot)-b= \cO(\varepsilon^2)$  on $[0,T_0]$ for our fixed $T_0$.
We recall that 
\[
\txtd u_c= (\nu u_c +\cF(u_c+u_s) )\txtd t + \sigma \txtd W_c. 
\]
If we define the error
\[
e:=b-\varepsilon^{-1} u_c(\varepsilon^{-2}\cdot) 
\]
we obtain after a short calculation
\[
\partial_t e = \frac{\nu}{\varepsilon^2} e + P_c\cF(b) 
- P_c\cF(\varepsilon^{-1} u(\varepsilon^{-2}\cdot) ).
\]
Taking the inner product with $e$ we further get
\begin{equation}
\label{e:est-e}
\frac12 \partial_t \|e \|^2  
= \frac12 \frac{\nu}{\varepsilon^2} \|e\|^2
+ \langle \cF_c(b)- \cF_c(\varepsilon^{-1} u(\varepsilon^{-2}\cdot)), e\rangle.
\end{equation}
We use with the short-hand notation $u^{(\varepsilon)}:=\varepsilon^{-1} u(\varepsilon^{-2}\cdot)$ and expand the cubic (using that on $\cN$ all norms are equivalent) to derive
\begin{eqnarray*}
\lefteqn{\langle \cF_c(b)- \cF_c(u^{(\varepsilon)}), e\rangle} \\
&\leq& \langle \cF_c(b)- \cF_c(u^{(\varepsilon)}_c)  , e\rangle
+  C \|e\|  \cdot ( \|u^{(\varepsilon)}_c\|^2 \|u^{(\varepsilon)}_s\|_X +  \|u^{(\varepsilon)}_c\| \|u^{(\varepsilon)}_s\|^2_X + \|u^{(\varepsilon)}_s\|^3_X) \\
&\leq &  -\delta \|e\|^4_X +  C \|e\|  \cdot ( \|u^{(\varepsilon)}_c\|^2 \|u^{(\varepsilon)}_s\|_X +  \|u^{(\varepsilon)}_c\| \|u^{(\varepsilon)}_s\|^2_X + \|u^{(\varepsilon)}_s\|^3_X) \\
&\leq &  -\frac12\delta \|e\|^4_X +  C \|u^{(\varepsilon)}_c\|^4   + C\|u_s^{(\varepsilon)}\|^4_X.
\end{eqnarray*}
Thus from \eqref{e:est-e} 
\[
\frac12 \partial_t \|e \|^2  
\leq \frac12 \frac{\nu}{\varepsilon^2} \|e\|^2
-\frac12\delta \|e\|^4_X +  C \|u^{(\varepsilon)}_c\|^4   + C\|u_s^{(\varepsilon)}\|^4_X
\]
and using a Gronwall type estimate 
we obtain for all $T\in[0,T_0]$ with constants depending on $T_0$
\[
\|e(T) \|^2  \leq C \|e(0)\|^2 + C \int_0^{T} ( \|u^{(\varepsilon)}_c\|^4   + C\|u_s^{(\varepsilon)}\|^4_X)~\txtd T =\cO(\varepsilon^4).
\]
Note that for $u_c$ due to the equivalence of norms, we can use the uniform bound in $\cH$,
while for $u_s $ we needed the bound on the  $L^4(0,T,X)$-norm. 
This finishes the proof of the last step and bounds the error on $\cN$. 
\qed


\section{Proof of the main results for non-negative $\nu$}\label{six}

We split the result in three cases.

\subsection{Case $1\gg\sigma\approx\nu >0$}\label{case1}
More precisely, we consider $\nu \in(0,\nu_0]$ for some sufficiently small $\nu_0$ and that the quotient $\sigma/\nu$ is bounded by positive constants from above and below.

In this case we observe with positive probability a finite-time Lyapunov exponent for the SPDE~\eqref{spde}. 
This means that we have instability shortly after the bifurcation. 
We recall that $\Omega_0$ was defined in Lemma~\ref{l1} and depends explicitly on the choice of the Brownian motion and thus on $\nu$. 
But the probability $\P(\Omega_0)\geq c>0$ is independent of the choice of the Brownian motion. 
\begin{theorem}\label{thm:main1}
Let $\{a(\omega)\}_{\omega\in\Omega}$ be the random fixed point of~\eqref{a1} and $\lambda_T$ be the finite-time Lyapunov exponent of the SPDE~\eqref{spde} with initial data $u_0=\varepsilon a(\omega)$. 
For all terminal times $T_0>0$ and all probabilities $p\in(0,1)$ 
there is a set $\Omega_p$ with probability larger than $p$ and a constant $C_p>0$
such that for $\omega\in \Omega_0\cap\Omega_p$ we have that 
\[ \lambda_{T\nu^{-1}}(\varepsilon a(\omega))> \frac{\nu}{4} -  C_p\frac{\varepsilon\nu}{T} \qquad \text{for all }T\in(0,T_0].\]
\end{theorem}


\begin{remark}
This theorem gives a positive lower bound on the Lyapunov exponent with positive probability for times between order $ 1/\sqrt{\nu} $ and $1/\nu$. We will see below that the lower bound on $T$ is mainly due to the uniform in time error estimate for the amplitude equation. 
We believe that this is only a technical restriction that can be solved by improving the approximation result for amplitude equations for smaller times, 
and we expect the finite time Lyapunov exponent to be positive for all times up to order $1/\nu$.  
 \end{remark}

We split the proof of Theorem \ref{thm:main1} into several steps. The main ideas are  the approximation of the SPDE~\eqref{spde} with the amplitude equation for $\varepsilon^2=\nu$,  Lemma~\ref{l1} and the control of the approximation error. 
As we start the SPDE in the rescaled random attractor,
we can apply these results, as $a=\cO(1)$ and thus $u_0=\varepsilon a = \cO(\varepsilon)$.

In this case the amplitude equation is given by %
\begin{align*}
\txtd b=[b-\cF_c(b)]~\txtd T + \frac{\sigma}{\nu} \txtd \beta_{\sqrt{\nu}}(T).
\end{align*}
Now we control the approximation error between the linearized SPDE and the linearized ODE.  To this aim we firstly introduce the slow scaling
$T=t\varepsilon^2$ and define $U$ via 
\[
u(t)= \varepsilon U(t\varepsilon^2).
\]
Let $v$ be the solution of the linearization of the SPDE around a solution $u$
\[ \partial_t v =Av+\nu v+ D\cF(u)v.
\]
On the slow scale $v(t)=\varepsilon V(t\varepsilon^2)$ 
we have (using that $D\mathcal{F}$ is quadratic)
\[
\partial_T V = \varepsilon^{-2} AV + V + D\cF(U)V. \]

Let $\varphi$ be the solution of the linearization of the amplitude equation around the attractor $\{a(\omega)\}_{\omega\in\Omega}$
\[
\partial_T \varphi = \varphi + D\cF_c(a) \varphi.
\]
We only consider initial conditions $V(0)=\varphi(0)\in\cN$ of order $1$ independent of $\varepsilon$.

The first crucial step is the following approximation result.
\begin{theorem}\label{error:linearization1}
For any probability $p\in(0,1)$ there is a set 
 $\Omega_p$ with probability larger than $p$ such that 
the error between the linearization of the SPDE~\eqref{spde1} with initial data $u_0=\varepsilon a(\omega)$ and of the amplitude equation~\eqref{a1} is bounded by $C\varepsilon$. 
\end{theorem}


\begin{proof}
We show in several steps that the following error bound holds on the set of large probability $\Omega_p$ 
\begin{align}\label{error:linearization}
\|V(T) - \varphi(T) \|_{\cH} \leq \|P_s V(T) +P_c V(T) -\varphi(T) \|_{\cH} \leq C \varepsilon,~~T\in[0,T_0].
\end{align}
To this aim we first prove 
\begin{equation}
\label{e:Vclaim}
  \| P_sV(T)\|_{\cH} =\cO(\varepsilon) 
\quad\text{and}\quad
\| P_sV\|_{L^2(0,T_0,\cH^{1/2})} =\cO(\varepsilon^{2}).   
\end{equation} 

We first consider $V$ and use standard energy-type estimates
to obtain 
\[
 \frac12 \partial_T\|V\|^2 =  \varepsilon^{-2} \langle AV, V \rangle  + \|V\|^2 + \langle D\cF(U)V , V \rangle \\
 \leq  \|V\|^2, 
\]
where we used the non-negativity of $A$ and \eqref{e:sign}. As  $V(0)=\cO(1)$ this yields a uniform $\cO(1)$-bound on $V$ and thus $P_cV$ in $\cH$ on $[0,T_0]$ (with  constants depending on $T_0$).

We have (using the short-hand  notation $V_s:=P_sV$ and $V_c:=P_cV$) 
\begin{align}
 \frac12 \partial_T \|V_s\|^2 = & \varepsilon^{-2} \langle AV_s, V_s \rangle  + \|V_s\|^2 + \langle P_sD\cF(U)V , V_s \rangle \nonumber\\
 \leq &  - c\varepsilon^{-2}\|V_s\|^2_{\cH^{1/2}}+\|V_s\|^2 + \langle P_sD\cF(U)V_c , V_s \rangle, \label{e:Vappr}
\end{align}
where we used the spectral properties of $A$ (Assumption~\ref{a}) and the sign condition on $D\cF$ from \eqref{e:sign}. 

The crucial step is to estimate the nonlinear term. This we can bound as follows 
\[
\langle P_sD\cF(U)V_c , V_s \rangle  \leq C \|U\|_X^2 \|V_c\|_X \|V_s\|_{\cH^{\alpha}}
\leq  C \varepsilon^2 \|U\|_X^4 \|V_c\|_X^2 + c \varepsilon^{-2}\|V_s\|^2_{\cH^{\alpha}},
\]
where we used $\varepsilon$-Young's inequality in the last step. 
Further, as shown above $V$ is $\cO(1)$ in $\cH$. 
Therefore we obtain that  $V_c$ is bounded in $X$ since all norms are equivalent on $\cN$. 
Consequently, we only need a bound on $\int_0^T \|U(S)\|_X^4~\txtd S$, which can be derived from the first step of the approximation result, Theorem~\ref{thm:aprox}. 
Namely, using 
$$ 
\int_0^{T_\varepsilon} \|u(t)\|^4_X~\txtd t =\cO(\varepsilon^2)
$$ 
we obtain
\begin{equation}
\label{e:intBouU}
  \int_0^{T_0} \|U(S)\|^4_X~\txtd S= 
\varepsilon^2\int_0^{T_\varepsilon} \|U(t\varepsilon^{2})\|^4_X~\txtd t 
=\varepsilon^{-2} \int_0^{T_\varepsilon} \|u(t)\|^4_X~\txtd t
=\cO(1).  
\end{equation}

Thus we can conclude from \eqref{e:Vappr} for two different universal constants $c>0$ and $C>0$ using that $\|\cdot\|_{\cH^\alpha}\leq \|\cdot\|_{\cH^{1/2}}$ 
\begin{align}\label{second}
    \frac{1}{2}\partial_T \|V_s\|^2 &\leq -c \varepsilon^{-2} \|V_s\|^{2}_{\cH^{1/2}} +\|V_s\|^2_{\cH^\alpha} + C\varepsilon^{2}\|U\|^4_X \|V_c\|^2_X + c\varepsilon^{-2}\|V_s\|^2_{\cH^\alpha}\nonumber\\
    & \leq -c\varepsilon^{-2}\|V_s\|^2_{\cH^{1/2}} + C \varepsilon^2\|U\|^4_X \|V_c\|^2_X. 
\end{align}
Consequently, recalling that $V_s(0)=0$ via a Gronwall type estimate we obtain for all $T\in[0,T_0]$ the inequality (with constants depending on $T_0$)
\[ 
\|V_s(T)\|^2 \leq  C \varepsilon^2 \int_0^{T_0} \|U(S)\|^4_X~\txtd S = \cO(\varepsilon^2),
\]
which means that $\|V_s\|_{\cH}=\cO(\varepsilon)$, as claimed.

For the second statement in \eqref{e:Vclaim} we get from~\eqref{second}  that 
\[ c\varepsilon^{-2} \|V_s\|^2_{\cH^{1/2}} 
\leq -\frac{1}{2}\partial_T \|V_s\|^2 + C\varepsilon^2 \|U\|^4_X \|V_c\|^2_X,  
\]
therefore by integration (recall $V_s(0)=0$) we derive
\[ \int_0^{T_0}\|V_s(S)\|^2_{\cH^{1/2}}~\txtd S 
\leq  -\frac{c \varepsilon^2}{2} \|V_s(T)\|^2 + C \varepsilon^4 \int_0^{T_0} \|U(S)\|^4_X\|V_c(S)\|^2_X~\txtd S. 
\]
As $\|V_s\|_\cH=\cO(\varepsilon)$ and $\int_0^{T_0} \|U(S)\|^4_X~\txtd S=\cO(1)$ we obtain 
\[
\|V_s\|_{L^2(0,T_0,\cH^{1/2})}=\cO(\varepsilon^2).
\]
We now focus on the bound for $\|V_c-\varphi\|$, which requires more work. We observe that $V_c-\varphi$ satisfies the equation
\[
\partial_T (V_c -\varphi) = V_c-\varphi + (D \cF_c (U) V - D\cF_c(a)\varphi),
\]
so we have to estimate
\begin{equation}
    \label{e:Vc-phi}
\frac{1}{2}\partial_T \|V_c-\varphi\|^2 =\| V_c-\varphi\|^2 + \langle D \cF_c (U) V - D\cF_c(a)\varphi, V_c-\varphi \rangle.
\end{equation}
Here the crucial term contains the nonlinearity 
\begin{eqnarray*}
  \langle D\cF_c(a)\varphi - P_c D\cF(U)V , \varphi-V_c\rangle_{\cN}
  &= &  - \langle  P_c D\cF(U)V_s ,  \varphi-P_cV\rangle_{\cN}\\
  && + \langle D\cF_c(a)\varphi - P_c D\cF(a)V_c , \varphi-P_cV\rangle_{\cN}\\
  &&+\langle  P_c [D\cF(a) - D\cF(U) ]V_c , \varphi-P_cV\rangle_{\cN},
\end{eqnarray*}
where the bound on $P_sV$ is needed in the space $X$, but here the integral bounds turn out to be sufficient.
We also rely on our  $\mathcal{O}(1)$-bounds on $\varphi$ and $V_c$.

We begin with the first term above which entails
\begin{align*}
\langle  P_c D\cF(U)V_s , \varphi-V_c\rangle_{\cN} 
&\leq  C \| U\|^2_X \|V_s\|_X \|\varphi-V_c\|_{\cN}\\
 &\leq C \|U\|^2_X \|V_s\|_{\cH^{1/2}}
 \|\varphi-V_c\|_{\cN} 
\\&
\leq C  \|V_s\|^2_{\cH^{1/2}} + \|U\|^4_X
 \|\varphi-V_c\|^2_{\cN} . 
 \end{align*}
 In the last step we used again Young's inequality.

The second term gives
\[  \langle D\cF_c(a)(\varphi-V_c),\varphi-V_c \rangle_{\cN} \leq C\|a\|_{\mathcal{N}}^2 \|\varphi-V_c\|^2_{\cN}.
\]

For the last one
we use that $P_c D\cF$ and $D\cF_c$ are the same on $\cN$, which can be seen by explicitly using the cubic.
\begin{align*}
\langle P_c [D\cF(a)-D\cF(U)] V_c,\varphi-V_c  \rangle_{\cN}
& \leq C \|a-U\|^2_X \|V_c\|_X \|\varphi-V_c\|_{\cN}
\\& \leq C \|a-U_c\|^2_{\cN} \|V_c\|_{\cN} \|\varphi-V_c\|_{\cN} 
+ C \|U_s\|^2_X \|V_c\|_{\cN} \|\varphi-V_c\|_{\cN}\\
&\leq C \|a-U_c\|^4_{\cN} \|V_c\|^2_{\cN} 
+ C \|U_s\|^4_X \|V_c\|^2_{\cN} + c\|\varphi-V_c\|^2_{\cN}.
\end{align*}
Regarding~\eqref{e:Vc-phi} and putting all the estimates together we infer that (with universal constants $C,\tilde{C}>0$)
\begin{align*}
\frac{1}{2}\partial_T \|V_c-\varphi\|^2 &\leq \|V_c-\varphi\|^2 + C \|V_s\|^2_{\cH^{1/2}} + \|U\|^4_X \|\varphi - V_c\|^2_{\cN} +\|a\|^2_{\cN}\|\varphi-V_c\|^2_{\cN} \\
& + C \|a-U_c\|^4_{\cN} \|V_c\|^2_{\cN} + C \|U_s\|^4_{X} \|V_c\|^2_{\cN} +C \|\varphi- V_c\|^2_{\cN}\\
& \leq \|V_c-\varphi\|^2_{\cN} ( \tilde{C}+\|U\|^4_X +\|a\|^2_{\cN})
+ C \|V_s\|^2_{\cH^{1/2}} + C \|a-U_c\|^4_{\cN}\|V_c\|^2_{\cN} +C\|U_s\|^4_X \|V_c\|^2_{\cN}\\
&\leq  I \cdot \|V_c-\varphi\|^2_{\cN} + C \cdot J,
\end{align*}
where we set 
\[
I:= \tilde{C}+\|U\|^4_X + \|a\|^2_{\cN}
\quad\text{and}\quad 
J:=\|V_s\|^2_{\cH^{1/2}} 
+  \|a-U_c\|^4_{\cN}\|V_c\|^2_{\cN}
+\|U_s\|^4_X \|V_c\|^2_{\cN}.
\]
Using Gronwall's inequality we get for $T\in[0,T_0]$
\begin{align*}
\|V_c(T)-\varphi(T)\|^2 \leq \|V_c(0)-\varphi(0)\|^2 
+\int_0^T J(S)~\txtd S \exp \Big(\int_0^T I(S)~\txtd S\Big).
\end{align*}
We now investigate the order of $J$.
First of all, since we start the SPDE in $u_0=\varepsilon a(\omega)$ we obtain due to Theorem~\ref{thm:aprox} 
\[
\|a(T)-U_c(T)\|_{\cH} 
= \varepsilon^{-1} \|\varepsilon a- u_c(\varepsilon^{-2} \cdot )\|_{\cH} =\cO(\varepsilon).
\]
Again we use the fact that all norms are equivalent on $\cN$.
Further, similar to \eqref{e:intBouU} by using the third step in Theorem~\ref{thm:aprox} we know that with $u_s(t)=\varepsilon U_s(t\varepsilon^2)$ 
\[ 
\int_0^{T_{0}} \|U_s(T)\|^4_X ~\txtd T
=
 \varepsilon^{-2} \int_0^{T_\varepsilon} \|u_s(t)\|^4_X~\txtd t 
=\cO(\varepsilon^2).
\]
Due to the above results, we have 
 pathwise bounds for 
 $\int_0^T J(S)~\txtd S \leq C\varepsilon^2$ on a set of probability going to $1$ for $C\to\infty$. Moreover,
\[ \exp\Big(\int_0^T I(S)~\txtd S \Big)
= \exp\Big(\int_0^T (\tilde{C}+\|U(S)\|^4_X + \|a(S)\|^2_{\cN}~\txtd S\Big)
\leq C 
\]
on a set of probability going to $1$ for $C\to \infty$. 
Together with the previous bound this gives another condition for the set $\Omega_p$.

In summary, this entails the following error bound on $\Omega_p$
\[
\|\varphi(T)-V_c(T)\|^2 \leq C \varepsilon^2 ~\text{ for } T\in[0,T_0].
\]
Putting all these deliberations together, proves the statement \eqref{error:linearization} on $\Omega_p$, i.e.\
\[
\|V(T)-\varphi(T)\|_{\cH} 
\leq \|V_s(T)\|_{\cH}+\|V_c(T)-\varphi(T)\|_{\cN} 
\leq C \varepsilon,~~T\in[0,T_0]. 
\]
Here we have to add another condition to $\Omega_p$, 
as $\|V_s\|_{\cH}\leq C\varepsilon^2$ uniformly in $T$ with probability going to $1$ if $C\to\infty$.
\qed
\end{proof}

Using this result we can proceed with the proof of Theorem~\ref{thm:main1}.
We first recall the definition of the FTLE
\begin{eqnarray*}
\lambda_{T\nu^{-1}}(\varepsilon a(\omega)) &=& \frac{\nu}T\ln( \sup\{\|v(T/\nu)\| \ : \ \|v(0)\|=1  \})\\
&=& \frac{\nu^{3/2}}
T\ln( \sup\{\|V(T)\| \ : \ \|V(0)\|= \varepsilon^{-1}  \})\\
&=& \frac{\nu}T\ln( \sup\{\|V(T)\| \ : \ \|V(0)\|=1  \}).
\end{eqnarray*}

Using \eqref{error:linearization} for the finite-time Lyapunov exponents of the SPDE we have on $\Omega_0\cap\Omega_p$ recalling Lemma~\ref{l1}
\begin{eqnarray*}
\|V(T)\|
&\geq  & \|\varphi(T)\| - \|V(T)-\varphi(T)\| 
\geq   \|\varphi(T)\| - C\varepsilon \\
&\geq  &   \exp\{(1-3\delta^2)T\}  - C\varepsilon >0,
\end{eqnarray*}
which is positive if $\varepsilon_0$ is sufficiently small. Here we can choose $\delta=\frac{1}{2}$ as in Lemma~\ref{l1}.

To proceed we use a simple estimate for the logarithm. It is known that there exists a positive constant $c>0$ such that $ \ln(1-x)\geq -c x $ for $0\leq x\leq\frac{1}{2}.$   Therefore as $\varepsilon \ll e^{Tc} $ we have that
\[
\ln(e^{cT}-\varepsilon) =\ln(e^{cT}(1-\varepsilon e^{-cT})) 
= cT + \ln(1-\varepsilon e^{-cT}) \geq cT - C\varepsilon e^{-cT}\geq cT - C\varepsilon.
\]
Thus we can conclude that on $\Omega_0\cap \Omega_p$ we can bound
\begin{eqnarray*}
\lefteqn{\lambda_{T\nu^{-1}}(\varepsilon a(\omega)) = 
\frac{\nu}T\ln( \sup\{\|V(T)\| \ : \ \|V(0)\|=1  \})}\\
&\geq  &  \frac{\nu}{T}\ln( \sup\{\|\varphi(T)\| - \|V(T)-\varphi(T)\| : \ \|V(0)\|=1  \})\\
&\geq  &  \frac{\nu}T\ln( \sup\{\|\varphi(T)\| - \|V(T)-\varphi(T)\| : \ \|v(0)\|=1, V(0)=\varphi(0)\in\cN  \})\\
&\geq  &  \frac{\nu}T\ln( \sup\{\|\varphi(T)\| - \varepsilon  : \ \|V(0)\|=1, V(0)=\varphi(0)\in\cN  \})\\
&\geq  &   \frac{\nu}T\ln( \exp\{(T+ \int_0^T D\cF_c(a(\theta_s\omega))~\txtd s\}  - C\varepsilon)\\
&\geq  &   \frac{\nu}T\ln( \exp\{(1-3\delta^2)T\}  - C\varepsilon)\\
& \geq & \nu (1-3\delta^2-C\varepsilon/T),
\end{eqnarray*}
and we finished the proof, as $\delta=1/2$.
\qed

\subsection{Case $1\gg\nu\gg\sigma>0$}\label{case2}

In this case the amplitude equation is given by 
\[
\txtd b=[b+\cF_c(b)]~\txtd T. 
\]
Here we consider the solution $b=0$ of the amplitude equation and let $u$ be the solution of SPDE with $u(0)=u_0=0$.
Let us remark that this is not precisely the setting of the approximation result of section \ref{five}, where the  amplitude equation would contain a small noise term $(\sigma/\nu)\txtd\beta$.   Here we simplify the proof by neglecting the small noise in the approximation. Thus we cannot use the full Theorem \ref{thm:aprox} here, but can still rely on all the bounds provided for $u$,  see Remark~\ref{r:a} below.

As before, let $V$ be the solution of the linearized SPDE
\[ 
\partial_T V =\varepsilon^{-2}AV+ V+ D\cF(U)V
\]
and thus
\[ \partial_T V_c = V_c+  P_cD\cF(U)V = V_c+ D\cF_c(U)(V_c+V_s).
\]
The linearization of the amplitude equation around 0 reduces to 
\[ \partial_T \varphi = \varphi + DF_c(0)\varphi ,\]
which gives 
\[ \partial_T \varphi =\varphi. \] 

The main result in this case reads as follows.

\begin{theorem}\label{thm:main2}
Let $\lambda_T$ be the finite-time Lyapunov exponent of the SPDE~\eqref{spde} with initial data $u_0=0$.  
For all probabilities $p\in(0,1)$ 
there is a set $\Omega_p$ with probability larger than $p$ and a constant $C_p>0$
such that for $\omega\in \Omega_p$ 
we have that  
\[ \lambda_{T\nu^{-1}}(0) > \nu - C_p\frac{\nu \varepsilon}{T}  .\]
\end{theorem}


\begin{proof}
Analogously to the previous case we have on a set $\Omega_p$ that
\begin{eqnarray*}
\lefteqn{\lambda_{T\nu^{-1}}(0) = \frac{\nu}T\ln( \sup\{\|v(T/\nu)\| \ : \ \|v(0)\|=1  \})}\\
&=& \frac{\nu^{3/2}}
T\ln( \sup\{\|V(T)\| \ : \ \|V(0)\|= \varepsilon^{-1}  \})\\
&=& \frac{\nu}T\ln( \sup\{\|V(T)\| \ : \ \|V(0)\|=1  \})\\
&\geq  &  \frac{\nu}{T}\ln( \sup\{\|\varphi(T)\| - \|V(T)-\varphi(T)\| : \ \|V(0)\|=1  \})\\
&\geq  &  \frac{\nu}T\ln( \sup\{\|\varphi(T)\| - \|V(T)-\varphi(T)\| : \ \|v(0)\|=1, V(0)=\varphi(0)\in\cN  \})\\
&\geq  &  \frac{\nu}T\ln( \sup\{\|\varphi(T)\| - \varepsilon  : \ \|V(0)\|=1, V(0)=\varphi(0)\in\cN  \})\\
& \geq & \frac{\nu}{T} \ln (\exp T -\varepsilon)  \\
&\geq & \nu  - C \frac{\nu \varepsilon}{T}.
\end{eqnarray*}
\qed
\end{proof}


In the previous computation we use similar to the first case that the error between the linearization of the SPDE~\eqref{spde1} with initial data $u_0=0$ and of the amplitude equation is of order $\cO(\varepsilon)$. This follows by similar computations as in~\ref{case1}, which we shortly sketch for the convenience of the reader. In contrast to~\eqref{case1}, the linerization of the amplitude equation changes, since we are now interested in the case $a(\omega)=0$. 

Therefore we get 
\[ \partial_T (V_c-\varphi) =V_c-\varphi +D\cF_c(U)(V_c+V_s), \]
which further leads to 
\begin{align*}
    \frac{1}{2} \partial_T \|V_c-\varphi\|^2 &=\|V_c-\varphi\|^2 + \langle D\cF_c(U)(V_c+V_s),V_c-\varphi \rangle\\
    & = \|V_c-\varphi\|^2 + \langle D\cF_c(U)V_c, V_c-\varphi  \rangle +\langle D\cF_c(U)V_s,V_c-\varphi \rangle\\
    & \leq \|V_c-\varphi\|^2 
    + c\|U\|^2_X \|V_c\|_{\cN} \|V_c-\varphi\|_{\cN} 
    +c\|U\|^2_X \|V_s\|_X \|V_c-\varphi\|_{\cN}\\
    & \leq \|V_c-\varphi\|^2 
    + c\|U\|^2_X \|V_c\|_{\cN} \|V_c-\varphi\|_{\cN} 
    +c\|U\|^2_X \|V_s\|_{\cH^{1/2}} \|V_c-\varphi\|_{\cN}\\
    & \leq c \|V_c-\varphi\|^2 
    +c\|U\|^4_X \|V_c\|^2_{\cN} + c\|V_s\|^2_{\cH^{1/2}} + c\|U\|^4_X\|V_c-\varphi\|^2_\cN\\
    & \leq c(1 +\|U\|^4_X)\|V_c-\varphi\|^2_\cN +c\|U\|^4_X\|V_c\|^2_\cN + c\|V_s\|^2_{\cH^{1/2}}\\
    & \leq c\|V_c-\varphi\|^2_\cN I +  cJ,
\end{align*}
where $I:=1 + \|U\|^4_X$ and $J:=\|U\|^4_X \|V_c\|^2_\cN+ \|V_s\|^2_{\cH^{1/2}}$ and $c$ stands for a universal constant which varies from line to line. 
Again, Gronwall's inequality on $[0,T_0]$ entails
\begin{align*}
\|V_c(T)-\varphi(T)\|^2 
\leq c \Big( \|V_c(0)-\varphi(0)\|^2 +c\int_0^T J(S)~\txtd S \Big) \cdot \exp \Big(c\int_0^T I(S)~\txtd S\Big).
\end{align*}
This gives $\|V_c(T)-\varphi(T)\| \leq C\varepsilon $ on a set of probability arbitrarily close to $1$ with constant depending probability.  In contrast to Case~\ref{case1} we need pathwise bounds on $J$ of order $\varepsilon^2$ which hold on a set of arbitrarily large probability since $\|U\|_{L^4(0,T_0,X)}=\cO(1)$ and that $\|V_s\|_{L^2(0,T_0,\cH^{1/2})}=\cO(\varepsilon^2)$. The exponent $\int_0^T I(S)~\txtd S$ can be bounded by a constant on a set of large probability. 

In conclusion we obtain on $[0,T_0]$ and on a set of probability arbitrarily close to $1$ 
\[ 
\|V(T)-\varphi(T)\|_\cH\leq  \|V_s(T)\|_{\cH} +\|V_c(T)-\varphi(T)\|_{\cH}\leq C\varepsilon.
\]
\begin{remark}\label{r:a}
Note that we do not rely on the approximation of the SPDE with the amplitude equation, which was not even established in this case. 
We only control the error term between the two linearizations. 
This is enough for our aims, since we start the SPDE and the amplitude equation in zero, in contrast to the previous case, and the $L^4(0,T;X)$ bound on $U$ established in Theorem~\ref{thm:aprox} suffices. 
In particular this follows from~\eqref{e:bound-e-int} and is independent of the amplitude equation.
\end{remark}

\subsection{Case: $\nu=0$, $1\gg\sigma>0$}\label{case3}


At the bifurcation point, we consider $\varepsilon=\sqrt{\sigma}$. Here the amplitude equation is
\[
\txtd b=P_c\cF(b)~\txtd T+\txtd \beta_\varepsilon(T).
\]
Therefore, we get 
\begin{align}\label{phi} \partial_T \varphi = D \cF_c(a)\varphi. 
\end{align}

The linearization of the SPDE~\eqref{spde} reads now as \[ \partial_t v = A v + D\cF(u)v, \]
which means that setting $v(t)=\varepsilon V(t\varepsilon^2)$ we obtain
\begin{align}\label{v} \partial_T V =\varepsilon^{-2} A V + D\cF(U)V. \end{align}
As in the previous cases we compute the error term between the two linearizations. 
\begin{theorem} The approximation order between the linearization of the SPDE~\eqref{v} and of the amplitude equation~\eqref{phi} with initial data $u_0=\varepsilon a(\omega)$ is bounded by  $C_p\varepsilon$ on a set $ \Omega_p$ with probability larger than $p$.
\end{theorem}
\begin{proof}
Since the linear term containing $\nu$ drops out, we compute new energy estimates. 
To get a $\cO(1)$ bound on $V$ we rely on the energy-estimate
\[ \frac{1}{2}\partial_T \|V\|^2 \leq \varepsilon^{-2}\langle A V, V \rangle + \langle D \cF(U)V,V\rangle, \]
which gives now due to~\eqref{e:boundF} 
\[ \frac{1}{2} \partial_T \|V\|^2 \leq \varepsilon^{-2} \langle AV, V \rangle\leq 0,\]
due to the non-negativity of $A$. As $V(0)=\cO(1)$ this yields a uniform $\cO(1)$ bound on $V$ in $\cH$ on $[0,T_0]$.  Due to the $\cO(1)$ bound on $V$ in $\cH$ we can also bound $V_c$ in $\cN$ in any norm.  

For $V_s$ we obtain as before that $\|V_s(T)\|_{\cH}=\cO(\varepsilon)$ and that $\|V_s\|_{L^2(0,T_0,\cH^{1/2})}=\cO(\varepsilon^2)$. 
This follows by the usual energy estimate regarding Assumption~\ref{a} and~\eqref{e:boundF} combined with the $\varepsilon$-Young inequality. To be more precise, the estimate is based on
\begin{align*}
    \frac{1}{2} \partial_T \|V_s\|^2& = \varepsilon^{-2}\langle A V_s, V_s\rangle + \langle P_s D\cF(U)V, V_s \rangle\\
    &\leq -C \varepsilon^{-2}\|V_s\|^2_{\cH^{1/2}}
    + C \|U\|^2_X \|V_c\|_X \|V_s\|_{\cH^\alpha}\\
    & \leq - C \varepsilon^{-2} \|V_s\|^2_{\cH^{1/2}} + C \varepsilon^2 \|U\|^4_X \|V_c\|^2_X + C \varepsilon^{-2}\|V_s\|_{\cH^\alpha}\\
    & \leq -C\varepsilon^{-2}\|V_s\|_{\cH^{1/2}} + C \varepsilon^2 \|U\|^4_X \|V_c\|^2_X.
\end{align*}

For $V_c$ and $\varphi$ we have 
$$ 
\partial_T V_c = D \cF_c(U)V 
\quad \text{ and } \quad 
\partial_T \varphi=D\cF_c(a) \varphi, $$
leading to
\[\partial_T (V_c-\varphi) = (D\cF_c(U)- D\cF_c(a)) V + D\cF_c(a) (V-\varphi) .
\]
For the difference, we estimate as follows. Here $c$ is a universal constant which varies from line to line. 
\begin{eqnarray*}
\lefteqn{\frac{1}{2}\partial_T \|V_c-\varphi\|^2}\\ &=&\langle (D\cF_c(U) -D\cF_c(a)) V , V_c-\varphi   \rangle + \langle D\cF_c(a) (V-\varphi), V_c-\varphi \rangle\\
& =& \langle (D\cF_c(U) -D\cF_c(a)) V , V_c-\varphi   \rangle + \langle D\cF_c(a) (V_c-\varphi), V_c-\varphi  \rangle + \langle D\cF_c(a) V_s, V_c-\varphi \rangle\\
& \leq& c\|U - a\|^2_{X} (\|V_s\|_{X} +\|V_c\|_{\cN}) \|V_c-\varphi\|  
+c\|a\|^2_{\cN} \|V_c-\varphi \|^2_{\cN} 
+ c\|a\|^2_{\cN} \|V_s\|_{\cH^\alpha} \|V_c-\varphi\|
\\
& \leq& c\|U_c-a\|^2_{\cN} (\|V_s\|_{\cH^\alpha} +\|V_c\|_{\cN}) \|V_c-\varphi\|_{\cN} 
+ c\|U_s\|^2_{X} (\|V_s\|_{\cH^\alpha}
+c\|V_c\|_{\cN} ) \|V_c-\varphi\|_{\cN}
\\ 
&& +\|a\|^2_{\cN} \|V_c-\varphi\|^2_{\cN} 
 +c \|a\|^4_{\cN} \|V_s\|^2_{\cH^\alpha} 
 + c \|V_c-\varphi\|^2_{\cN}
 \\
& \leq & c \|U_c-a\|^4_{\cN} \|V_c-\varphi\|^2_{\cN} + c \|V_s\|^2_{\cH^{1/2}} 
+ c \|U_c-a\|^4_{\cN} \|V_c\|^2_{\cN} 
+ c \|V_c-\varphi\|^2_{\cN}
\\
&& + c\|U_s\|^4_X \|V_c-\varphi\|^2_{\cN} 
+  c\|U_s\|^4_X \|V_c\|^2_{\cN} 
+ c\|a\|^2_{\cN} \|V_c-\varphi\|^2_{\cN} 
 +c\|a\|^4_{\cN} \|V_s\|^2_{\cH^{1/2}}.  
\end{eqnarray*}
Thus
\[
\partial_T \|V_c-\varphi\|^2 \leq c I \|V_c-\varphi\|^2_{\cN} + c J,
\]
where 
\[I:= 1 +\|U_c-a\|^4_{\cN} +\|U_s\|^4_X +\|a\|^2_{\cN}
\] 
and 
\[J := \|V_s\|^2_{\cH^{1/2}} + \|U_c-a\|^4_{\cN} \|V_c\|^2_{\cN} +\|U_s\|^4_{X}\|V_c\|^2_{\cN} + \|a\|^4_{\cN}\|V_s\|^2_{\cH^{1/2}}.
\]
Using
Gronwall's inequality as before we obtain
\begin{align*}
\|V_c(T)-\varphi(T)\|^2 \leq \Big( \|V_c(0)-\varphi(0)\|^2 +c\int_0^T J(S)~\txtd S \Big)\cdot \exp \Big(c\int_0^T I(S)~\txtd S\Big).
\end{align*}
Now we use again the $\cO(1)$ bounds on $V_c$ and $a$ and the $\cO(\varepsilon^2)$-bounds for $\|U_s\|_{L^4(0,T_0,X)}$ and $\|V_s\|_{L^2(0,T_0,\cH^{1/2})}$
together with Theorem~\ref{thm:aprox} that yields
\[
\|a(T)-U_c(T)\|_{\cH} 
= \varepsilon^{-1} \|\varepsilon a- u_c(\varepsilon^{-2} \cdot )\|_{\cH} =\cO(\varepsilon).
\]
In contrast to case \ref{case1} we only need  pathwise bounds on $J$ and $a$ of order one, which hold on a set of probability arbitrarily close to $1$. There is no need for $a$ being small.     

Moreover, a bound by a constant of the exponent $\int_0^T I(S)~\txtd S$ holds as before only on some set of probability arbitrarily close to $1$. 

Thus we finally conclude  
$\|V(T)-\varphi(T)\|_\cN\leq \varepsilon$ for all $T\in[0,T_0]$ on a set of probability  arbitrarily close to $1$. 
\qed
\end{proof}

\begin{theorem}\label{thm:main3} Let $\{a(\omega)\}_{\omega\in\Omega}$ be the random fixed point of~\eqref{a2} and $\lambda_T$ be the finite-time Lyapunov exponent of the SPDE~\eqref{spde} with initial data $u_0=\varepsilon a(\omega)$.
For all probabilities $p\in(0,1)$ 
there exists sets $\Omega_p$ and $\tilde{\Omega}_p$ with probability larger than $p$ and a constant $C_p>0$
such that for $\omega\in \Omega_p\cap\tilde{\Omega}_p$ we have that 
\[\lambda_{T\varepsilon^{-1}}(\varepsilon a(\omega))\leq -C_p\varepsilon +\varepsilon^2 \frac{e^{cT}}{T} \leq -\tilde{c}<0. \]
\end{theorem}

\begin{proof}
Similar to~\ref{case1} we infer
\begin{eqnarray*}
\lefteqn{\lambda_{T\varepsilon^{-1}}(\varepsilon a(\omega)) = \frac{\varepsilon}T\ln( \sup\{\|V(T)\| \ : \ \|V(0)\|=1  \})}\\
& \leq & \frac{\varepsilon}{T}  \ln\sup (\{\|\varphi(T)\| + \|V(T)-\varphi(T)\| : \|V(0)\|=1\})\\
&\leq   & \frac{\varepsilon}T\ln( \exp\{\int_0^T D\cF_c(a(\theta_s\omega))~\txtd s\}  +\varepsilon).
\end{eqnarray*}
The upper bound is clear as long as the attractor does not spend too much time in zero. To exclude this possibility, we need a lower bound on the probability of a set $\widetilde{\Omega}$, where 
\[
 \int_0^T D\cF_c(a(\theta_s\omega))~\txtd s \geq -cT.
\]
This is provided in Lemma~\ref{birkhoff}. Regarding this we easily derive on a set of large probability $\Omega_p\cap \tilde{\Omega}_p$ that
\[
\ln (e^{-cT} +\varepsilon) = \ln (e^{-cT}(1+\varepsilon e^{cT}) ) =-cT +\ln (1+\varepsilon e^{cT}) \leq - cT + \varepsilon e^{cT}.
\]
This further leads to 
\[ \lambda_{T\varepsilon^{-1}} (\varepsilon a(\omega))  \leq - c\varepsilon + \varepsilon^2 \frac{e^{cT}}{T},\]
which proves the statement.
\qed\end{proof}


\subsection{Case: $1\gg\sigma\gg  \nu >0$}\label{case4}

This situation can be dealt with similar to Case~\ref{case3} using the amplitude equation
\begin{align}\label{a3}
\txtd \tilde{b}=[\frac{\nu}{\sigma}\tilde{b} + P_c\cF(\tilde{b})]~\txtd T+\txtd \beta_\varepsilon(T),
\end{align}
and its linearization
\[ \partial_T \tilde{\varphi} =\frac{\nu}{\sigma}\tilde{\varphi} +D\cF_c(a)\tilde{\varphi}.   \]
Since the difference between~\eqref{a3} and~\eqref{a2} is of order $\cO(\frac{\nu}{\sigma})$, the following statement can be obtained analogously to Case~\ref{case3}.
\begin{theorem}\label{thm:main4} Let $\{a(\omega)\}_{\omega\in\Omega}$ be the random fixed point of~\eqref{a3} and $\lambda_T$ be the finite-time Lyapunov exponent of the SPDE~\eqref{spde} with initial data $u_0=\varepsilon a(\omega)$.
For all probabilities $p\in(0,1)$ 
there exists sets $\Omega_p$ and $\tilde{\Omega}_p$ with  probability larger than $p$ and a constant $C_p>0$
such that for $\omega\in \Omega_p\cap\tilde{\Omega}_p$ we have that 
\[\lambda_{T\varepsilon^{-1}}(\varepsilon a(\omega))\leq \Big(- C_p +\frac{\nu}{\sigma}\Big) \varepsilon +  \varepsilon \Big(\frac{\nu}{\sigma}\Big)^2 \frac{e^{c-\frac{\nu}{\sigma}}}{T}  \leq -\tilde{c}<0. \]
\end{theorem}

\begin{remark}
For a fixed $T$ the bound on the FTLE is negative with  since $\sigma\gg \nu$. As $\varepsilon=\sqrt{\sigma}$
we obtain negative FTLE for times $t$ with 
$ (\nu/\sigma)^2 \ll  t \leq T_0 /\sqrt{\sigma} $.
\end{remark}

\begin{proof}
We only give a sketch of the proof, since this is similar to Case~\ref{case3}. Regarding the computations in Case~\ref{case3} we infer on a set of probability almost $1$ that \[\|V(T)-\tilde{\varphi}(T)\| \leq C_p\big( (\varepsilon+ \frac{\nu}{\sigma})^2 \big).
\] 
This follows from the bound on $J$, which is now determined by 
\[ \|\tilde{a}(T)-U_c(T) \|_{\cN}=\cO \Big(\big( \varepsilon+ \frac{\nu}{\sigma}\big)^2\Big),
\]
where $\tilde{a}$ is the attractor of~\eqref{a3}. Note that in Case~\ref{case3} the order of $J$ was in lowest order determined by $\|V_s\|_{L^2(0,T;\cH^{1/2})}=\cO(\varepsilon)$ and the other terms were higher order. 

Therefore the lower bound for the FTLE  for $\omega\in \Omega_p\cap\tilde\Omega_p$ (as in Case~\ref{case3}) results in 
\begin{eqnarray*}
\lefteqn{\lambda_{T\varepsilon^{-1}}(\varepsilon a(\omega)) = \frac{\varepsilon}T\ln( \sup\{\|V(T)\| \ : \ \|V(0)\|=1  \})}\\
& \leq & \frac{\varepsilon}{T}  \ln\sup (\{\|\varphi(T)\| + \|V(T)-\varphi(T)\| : \|V(0)\|=1\})\\
&\leq   & \frac{\varepsilon}T\ln\Big( \exp\Big\{ \frac{\nu}{\sigma} T +  \int_0^T D\cF_c(a(\theta_s\omega))~\txtd s\Big\}  +C\Big(\frac{\nu}{\sigma}\Big)^2 \Big).
\end{eqnarray*}
Utilizing our Birkhoff-Lemma~\ref{birkhoff} this entails for $\omega\in\Omega_p\cap\tilde{\Omega}_p$ that
\[ \lambda_{T\varepsilon^{-1}} (\varepsilon a(\omega))\leq \varepsilon \Big(-c +\frac{\nu}{\sigma} \Big) + \varepsilon \Big(\frac{\nu}{\sigma}\Big)^2 \frac{e^{c-\frac{\nu}{\sigma}}}{T}, \]
as claimed.
    \qed
\end{proof}

\section{Examples}\label{seven}

In this section we list a few well known examples of SPDEs that satisfy the assumptions made in Section~\ref{assumptipns}. The approximation via amplitude equations has been discussed for example in~\cite{BlHa:04,Bl:05,Bl:07}, but either in a slightly different setting or with a different statement of the main error estimate. 

Combing this approach with finite-time Lyapunov exponents, our main results (Section~\ref{six}) provide a partial bifurcation analysis for these SPDEs. 

\subsection{Allen-Cahn}
This case is well studied including the properties of FTLE.  See for example~\cite{C,BlEnNe:21}. 

To be more precise, we 
consider the stochastic Allen-Cahn equation with Dirichlet boundary conditions on $[0,\pi]$ given by 
\begin{align}\label{ac}
\begin{cases}
    \txtd u = ((\Delta+1) u + \nu u -u^3)~\txtd t + \sigma~\txtd W_t\\
    u(0)=u_0\in\cH,~ u(0)=u(\pi)=0.
    \end{cases}
\end{align}
The solution operator is an order-preserving random dynamical system on $\cH:=L^2([0,\pi])$ which has a singleton attractor~\cite{C,FlGeSch:17}.
However in~\cite{BlEnNe:21} it was justified that a bifurcation can be detected using finite-time Lyapunov exponents and the monotonicity induced by order preservation was crucial in all these results. 
The statements obtained in Section~\ref{six} recover these results using only the approximation via the amplitude equations~\eqref{a1} and~\eqref{a2}. 

For the assumptions, it is well known that $A=\Delta+1$ generates an analytic semigroup in $\cH=L^2([0,\pi])$ that satisfies Assumption \ref{a} with a one-dimensional kernel $\mathcal{N}$ spanned by the sine. 
For the cubic nonlinearity $\cF(u)=-u^3$ we can directly check \eqref{e:stableF} and thus Assumtion \ref{f} using $X=L^4([0,\pi])$. 
Moreover the regularity of the stochastic convolution $Z$ is easy to establish if $W$ is sufficiently regular. For example, for space-time white noise  one has that $Z$ is jointly continuous in $t$ and $x$. 



\subsection{Swift-Hohenberg}
We consider the SPDE with Neumann boundary conditions on the interval $[0,k\pi]$ for some integer $k\geq1$ given by
\begin{align}\label{sh}
\begin{cases}
    \txtd u =[- (1+\Delta)^2 u  +\nu u - u^3]~\txtd t + \sigma~\txtd W_t\\
    u(0)=u_0\in\cH,~ u'(0)=u'(k\pi).
    \end{cases}
\end{align}
This equation is a toy model for the first convective instability in Rayleigh-Benard convection and a well studied in pattern formation. See \cite{CrHo:93}.

The differential operator $A=- (1+\Delta)^2$ 
as in the previous example generates  an analytic semigroup in $\cH=L^2([0,k\pi])$ that satisfies Assumption \ref{a} with a one-dimensional kernel $\mathcal{N}$ spanned by the sine. 
The nonlinearity $\cF$ is the same as in the previous example with $X=L^4([0,k\pi])$. Furthermore, the regularity of $Z$ works in a similar way, and we can allow even noise less regular than space-time white noise, as the differential operator is fourth order instead of second.

For previous results on the approximation via amplitude equations for Swift-Hohenberg see~\cite{BlHa:04} or \cite{BlMPSch:01}. However, since this random dynamical system is not order-preserving, to our knowledge there are no previous results on Lyapunov exponents, while we obtain a qualitative change in the behaviour of FTLEs at $\nu=0$.

\subsection{Surface growth model}

Our final example is a model from surface growth, which was originally proposed by \cite{LaDS:91},
\[
\txtd h = [-\Delta^2 h - \mu \Delta h + \nabla \cdot (\nabla h |\nabla h|^2)]\txtd t + \sigma  \txtd W
\]
subject to Neumann boundary conditions on $[0,L]$ and in a moving frame $\int_0^L h(tx) ~\txtd x =0$. Here the fluctuations are  space-time white noise given by fluctuations in the deposited material, but as the equation is posed in a moving frame, we need to remove the first Fourier mode from the noise.  

These type of equations are often studied with additional quadratic terms and periodic boundary conditions. See \cite{Ag:15} or for a general survey \cite{BS:95}.

Again, the linear operator 
generates an analytic semigroup in $\cH=L^2([0,L])$ satisfying Assumption \ref{a}. The nonlinearity 
$\cF(h)=\nabla \cdot (\nabla h |\nabla h|^2)$ satisfies Assumption \ref{f} with $X=W^{1,4}([0,L])$. This follows by similar computations than for the plain cubic taking into account the derivatives. 
For the regularity of $Z$ it is well known that for space-time white noise the spatial derivative $\partial_x Z$ is continuous in space and time.  

This equation fits into the setting of amplitude equations with cubic nonlinearities and it was briefly mentioned in \cite{BloemkerMohammed} where degenerate noise was studied.

\begin{remark}
Let us remark that this equation does not fit directly to our results. 
The change of stability appears here at $\mu_0= \pi^2/L^2$ and we could fix $\mu=\mu_0+\nu$ in order to apply our results. 
But then we need a small modification of the approximation result, as we have $\nu\Delta h$ instead of $\nu u$. 
The estimate via amplitude equations should hold with the same order of error in a similar way.
Nevertheless, some of the methods in the technical bounds in Step 2  do change, but  we do not give details here.  
The limiting SDE does not change and the final bounds on the FTLEs remain the same. 
\end{remark}

{\bf Acknowledgements.} The authors thank the two referees for the valuable comments.

\end{document}